\documentclass{amsart}
\usepackage[utf8]{inputenc}

\usepackage{amssymb}
\usepackage{amsmath}
\usepackage{amsthm}
\usepackage{mathtools}
\usepackage{multirow}
\usepackage{thmtools}
\usepackage{nicefrac}

\usepackage{coxeter}

\usepackage[colorlinks=true, citecolor=blue]{hyperref}
\usepackage{url}

\usepackage{todonotes}

\theoremstyle{plain}

\newtheorem{theorem}{Theorem}[section]
\newtheorem*{theorem*}{Theorem}
\newtheorem{corollary}[theorem]{Corollary}
\newtheorem{lemma}[theorem]{Lemma}
\newtheorem{proposition}[theorem]{Proposition}

\theoremstyle{definition}

\theoremstyle{remark}
\newtheorem{remark}[theorem]{Remark}

\DeclareMathOperator{\euclidean}{\mathbb{E}}
\DeclareMathOperator{\spherical}{\mathbb{S}}
\DeclareMathOperator{\hyperbolic}{\mathbb{H}}
\DeclareMathOperator{\faces}{\mathcal{F}}
\DeclareMathOperator{\simplex}{\Delta}

\newcommand{\Simp}[1]{\mathsf{Simp^#1}}
\newcommand{\Cubes}{\mathsf{Cubes}}

\title[Lann\'er diagrams and Coxeter polytopes]{Lann\'er diagrams and combinatorial properties of compact hyperbolic Coxeter polytopes}
\author{Stepan Alexandrov}
\address{Department of Discrete Mathematics, Moscow Institute of Physics and Technology}
\email{aleksandrov.sa@phystech.edu}
\subjclass[2020]{51F15, 52B11}

\begin{document}

\begin{abstract}
	In this paper we study \(\times_0\)-products of Lann\'er diagrams. We prove that every \(\times_0\)-product of at least four Lann\'er diagrams with at least one diagram of order \(\geqslant 3\) is superhyperbolic. As a corollary, we obtain that known  classifications exhaust all compact hyperbolic Coxeter polytopes that are combinatorially equivalent to products of simplices.
    
	We also consider compact hyperbolic Coxeter polytopes whose every Lann\'er subdiagram has order \(2\). The second result of this paper slightly improves recent Burcroff's upper bound on the dimension of such polytopes to \(12\).
\end{abstract}

\maketitle

\section{Introduction}

A convex polytope is called a \emph{Coxeter polytope} if its dihedral angles are all integer submultiples of \(\pi\). Compact Coxeter polytopes in \(\spherical^d\) and \(\euclidean^d\) were classified by Coxeter in \cite{C34}. Vinberg initiated the study of such polytopes in \(\hyperbolic^d\) and proved in \cite{V84} that there are no compact Coxeter polytopes in \(\hyperbolic^{\geqslant 30}\). Examples are known only in \(\hyperbolic^{\leqslant 8}\), the unique known example in \(\hyperbolic^8\) and both known examples in \(\hyperbolic^7\) are due to Bugaenko (\cite{B92}).

Thus, there are two very hard long-standing open problems. The first one is the construction of new hyperbolic Coxeter polytopes, especially higher-dimensional ones. And the second one is the classification of such polytopes.

Generally speaking, there are two different approaches to both problems: classification of finite-volume Coxeter polytopes of some certain combinatorial types (see \cite{K74, E96, T07, FT08, FT09, JT18, Bur22, MZ22a, MZ22b}) and the theory of arithmetic hyperbolic reflection groups (see \cite{V72, B16, B17, BP18, B19, B20}). In particular, in the context of arithmetic and quasi-arithmetic reflection groups several authors constructed new Coxeter polytopes as faces or reflection centralizers of some higher dimensional polytopes (see \cite{Bor87, All06, All13, BK21, BBKS21}).

This article is focused on the combinatorial approach, so let us give a brief summary of the results on the classification of compact hyperbolic Coxeter polytopes of certain combinatorial properties. A complete classification of Coxeter polytopes in \(\hyperbolic^2\) was obtained by Poincaré (\cite{P82}). Andreev (\cite{And70a, And70b}) described all Coxeter polytopes in \(\hyperbolic^3\). Compact Coxeter simplices were classified by Lann\'er (\cite{L50}). Kaplinskaya (\cite{K74}) used this classification to list all compact simplicial prisms. Esselmann (\cite{E96}) used Gale diagrams to list the remain compact polytopes in \(\hyperbolic^d\) with \(d + 2\) facets. Tumarkin (\cite{T07}) improved this technique and listed all compact polytopes in \(\hyperbolic^d\) with \(d + 3\) facets. All cubes were classified by Jacquemet and Tschantz (\cite{JT18}). Very recently and independently, Burcroff (\cite{Bur22}) and Ma \& Zheng (\cite{MZ22a, MZ22b}) listed all compact Coxeter polytopes in \(\hyperbolic^d\) with \(d + 4\) facets for \(d = 4, 5\).

The work is based on the author's bachelor thesis (\cite{Ale21}) supervised by Nikolay Bogachev.

\subsection{Classification of compact Coxeter products of simplices}

First of all, we should provide some definitions. Each Coxeter polytope can be described by its Coxeter diagram. Such a diagram contains information about the angles and distances between every pair of facets. Coxeter diagrams of the compact simplices in hyperbolic spaces were listed by Lann\'er (\cite{L50}) and are now called Lann\'er diagrams. They have an important property. Consider a compact hyperbolic Coxeter polytope and a minimal set of its facets with empty intersection. The subdiagram that corresponds to the set is a Lann\'er diagram.

The Lann\'er diagrams play an important role in many classifications as they are ``unfriendly'' to each other. These diagrams often form so-called superhyperbolic diagrams, which are not contained in any diagram of a hyperbolic Coxeter polytope. The first theorem provides a result of this type.

Denote by \(\mathcal{L}_{k_1} \times_0 \dots \times_0 \mathcal{L}_{k_n}\) the set of all Coxeter diagrams generated\footnote{The exact definition is given at the beginning of Section~ \ref{section: product of simplices}.} by pairwise disjoint Lann\'er diagrams of orders \(k_1, \dots, k_n\) and containing no other Lann\'er subdiagrams. Let us introduce the notation for some families of compact hyperbolic Coxeter polytopes:
\begin{itemize}
	\item \(\Simp{*}\) for all products of simplices;
	\item \(\Simp{k}\) for all products of \(k\) simplices;
	\item \(\Cubes\) for all cubes (not necessarily \(3\)-dimensional).
\end{itemize}

\begin{restatable}{maintheorem}{TheoremA} \label{theorem: products of lanner diagrams}
    Let \(n \geqslant 4\) and \(2 \ne k_1 \geqslant \dots \geqslant k_n = 2\). Every diagram contained in the set \(\mathcal{L}_{k_1} \times_0 \dots \times_0 \mathcal{L}_{k_n}\) is superhyperbolic.
\end{restatable}

As a simple corollary of this theorem, we obtain the following.

\begin{restatable}{maintheorem}{TheoremB} \label{theorem: products of simplices}
	\(\Simp{*} = \Simp{1} \cup \Simp{2} \cup \Simp{3} \cup \Cubes\).
\end{restatable}

\subsection{Compact 3-free Coxeter polytopes}

Now let us consider the polytopes with diagram containing no Lann\'er subdiagrams of order \(\geqslant 3\). These are exactly the polytopes with the following property: every set of facets with an empty intersection contains a pair of disjoint facets. Such polytopes are called \(3\)-free polytopes. For example, cubes satisfy this property, so the Coxeter diagram of a cube does not contain a Lann\'er subdiagram of order \(\geqslant 3\). Another example that satisfies this property is the family of compact right-angled polytopes in hyperbolic spaces (the reason is the structure of their diagrams). It is known that there are no Coxeter cubes in \(\hyperbolic^{\geqslant 6}\) (\cite{JT18}) and that there are no compact right-angled polytopes in \(\hyperbolic^{\geqslant 5}\) (\cite{PV05}). Recently Burcroff in \cite{Bur22} used Vinberg's methods to  estimate the dimension of such polytopes. We slightly improved this estimation.

\begin{restatable}{maintheorem}{TheoremC} \label{theorem: 3-free}
	Every Coxeter diagram of a compact Coxeter polytope in \(\hyperbolic^{\geqslant 13}\) contains a Lann\'er diagram of order \(\geqslant 3\).
\end{restatable}

\subsection*{Acknowledgements}

The author is grateful to Nikolay Bogachev for his supervision and remarks, and to Anna Felikson for maintaining her \href{https://www.maths.dur.ac.uk/users/anna.felikson/Polytopes/polytopes.html}{webpage} on hyperbolic Coxeter polytopes.

\subsection*{Funding}
The work was supported by the Theoretical Physics and Mathematics Advancement Foundation ``BASIS''.

\section{Preliminaries}

\subsection{Abstract diagrams}

A \emph{diagram} is a graph with positive real weights on the edges. The \emph{order} \(|S|\) of a diagram \(S\) is the number of vertices of the graph. A~\emph{subdiagram} of a diagram \(S\) is a diagram obtained from \(S\) by erasing some vertices. Consider a diagram \(S\). A diagram \emph{generated} by subdiagrams \(S_1, \dots, S_k\) of \(S\) and vertices \(v_1, \dots, v_l\) of \(S\) is a subdiagram \(\langle S_1, \dots, S_k, v_1, \dots, v_l \rangle\) of \(S\) obtained from~ \(S\) by erasing every vertex \(v\) that is not contained in any \(S_i\) and is not equal to any~ \(v_j\).

Let \(S\) be a diagram. Consider a symmetric matrix \((g_{ij})\) such that \(g_{ij}\) equals one if \(i = j\), zero if \(v_i v_j\) is not an edge of the diagram \(S\), and \(-w_{ij}\) if \(w_{ij}\) is the weight of the edge \(v_i v_j\). Such a matrix \(G(S) = (g_{ij})\) is called the \emph{Gram matrix} of the diagram \(S\).

We say that a diagram has some property if its Gram matrix has the same property (e.g., positive definiteness). A diagram has the same determinant and signature as its Gram matrix.

A diagram is said to be \emph{elliptic} if it is positive definite, \emph{parabolic} if it is positive semidefinite and not elliptic, and \emph{hyperbolic} if it is indefinite with the negative inertia index equals one.

A \emph{product} of diagrams \(S_1\) and \(S_2\) is a diagram whose vertex set is the disjoint union of the vertex sets of \(S_1\) and \(S_2\) and whose edge set is the union of the edge sets of \(S_1\) and \(S_2\) (informally speaking, we draw two diagrams side by side). The Gram matrix of such diagram is equal to \(G(S_1) \oplus G(S_2)\) up to simultaneous permutation of rows and columns. A diagram is \emph{connected} if it is not a product of some other non-empty diagrams.

Obviously, every elliptic diagram is a product of some connected elliptic diagrams. Every parabolic diagram is a product of some connected elliptic diagrams and some (at least one) connected parabolic diagrams.

\begin{proposition}
	A hyperbolic diagram does not contain a subdiagram that is a product of two hyperbolic diagrams.
\end{proposition}

\subsection{Coxeter diagrams}

A diagram is called a \emph{Coxeter diagram} if each of its weights is either \(\geqslant 1\) or equal to \(\cos\left(\frac{\pi}{m}\right)\) for some \(m \geqslant 3\). Such diagrams are usually drawn as follows. If the weight of an edge \(v_i v_j\) is greater than one, then a dashed edge is drawn connecting \(v_i\) and \(v_j\). If the weight of an edge \(v_i v_j\) is equal to one, then a bold edge is drawn. If the weight of an edge \(v_i v_j\) is equal to \(\cos(\frac{\pi}{m})\), then a \((m - 2)\)-fold edge or a simple edge with label \(m\) is drawn. We say that a vertex \(v\) is \emph{joined} with a vertex \(u\) if they are joined by any edge other than a \(2\)-labeled one.

\begin{theorem}[\cite{C34}]
	Connected elliptic and parabolic diagrams are listed in Table \ref{table: connected elliptic coxeter diagrams} and Table \ref{table: connected parabolic coxeter diagrams}.
\end{theorem}

\begin{table}
    \centering
    \begin{tabular}{r l r l}
        \(A_n\)\small\((n \geqslant 1)\) &
        \begin{coxeter}
            \node{1}{(0, 0)}{} \node{2}{(1, 0)}{}
            \hnode{3}{(2, 0)} \hnode{4}{(3, 0)}
            \node{5}{(4, 0)}{} \node{6}{(5, 0)}{}
            \line{1}{2}{} \line{2}{3}{}
            \line{4}{5}{} \line{5}{6}{}
            \dots{3}{4}
        \end{coxeter} &
        \(B_n = C_n\)\small\((n \geqslant 2)\) &
        \begin{coxeter}
            \node{1}{(0, 0)}{} \node{2}{(1, 0)}{}
            \hnode{3}{(2, 0)} \hnode{4}{(3, 0)}
            \node{5}{(4, 0)}{} \node{6}{(5, 0)}{}
            \line{1}{2}{} \line{2}{3}{}
            \line{4}{5}{} \lline{5}{6}
            \dots{3}{4}
        \end{coxeter} \\ [5pt]
        
        \(D_n\)\small\((n \geqslant 4)\) &
        \begin{coxeter}
            \node{1}{(0, 0)}{} \node{2}{(1, 0)}{}
            \hnode{3}{(2, 0)} \hnode{4}{(3, 0)}
            \node{5}{(4, 0)}{} \node{6}{(5, .5)}{}
            \node{7}{(5, -.5)}{}
            \line{1}{2}{} \line{2}{3}{}
            \line{4}{5}{} \line{5}{6}{} \line{5}{7}{}
            \dots{3}{4}
        \end{coxeter} &
        \(G_2^{(m)}\) &
        \begin{coxeter}
            \node{1}{(0, 0)}{} \node{2}{(1, 0)}{}
            \line{2}{1}{$m$}
        \end{coxeter} \\
    
        \(F_4\) &
        \begin{coxeter}
            \node{1}{(0, 0)}{} \node{2}{(1, 0)}{}
            \node{3}{(2, 0)}{} \node{4}{(3, 0)}{}
            \line{1}{2}{} \lline{2}{3} \line{3}{4}{}
        \end{coxeter} &
        \(E_6\) &
        \begin{coxeter}
            \node{1}{(0, .5)}{} \node{2}{(1, .5)}{}
            \node{3}{(2, .5)}{} \node{4}{(2, -.5)}{}
            \node{5}{(3, .5)}{} \node{6}{(4, .5)}{}
            \line{1}{2}{} \line{2}{3}{} \line{3}{4}{}
            \line{3}{5}{} \line{5}{6}{}
        \end{coxeter} \\ [10pt]
    
        \(E_7\) &
        \begin{coxeter}
            \node{1}{(0, .5)}{} \node{2}{(1, .5)}{}
            \node{3}{(2, .5)}{} \node{4}{(2, -.5)}{}
            \node{5}{(3, .5)}{} \node{6}{(4, .5)}{}
            \node{7}{(5, .5)}{}
            \line{1}{2}{} \line{2}{3}{} \line{3}{4}{}
            \line{3}{5}{} \line{5}{6}{} \line{6}{7}{}
        \end{coxeter} &
        \(E_8\) &
        \begin{coxeter}
            \node{1}{(0, .5)}{} \node{2}{(1, .5)}{}
            \node{3}{(2, .5)}{} \node{4}{(2, -.5)}{}
            \node{5}{(3, .5)}{} \node{6}{(4, .5)}{}
            \node{7}{(5, .5)}{} \node{8}{(6, .5)}{}
            \line{1}{2}{} \line{2}{3}{} \line{3}{4}{}
            \line{3}{5}{} \line{5}{6}{} \line{6}{7}{}
            \line{7}{8}{}
        \end{coxeter} \\ [5pt]
    
        \(H_3\) &
        \begin{coxeter}
            \node{1}{(0, 0)}{} \node{2}{(1, 0)}{}
            \node{3}{(2, 0)}{}
            \line{1}{2}{} \llline{2}{3}
        \end{coxeter} &
        \(H_4\) &
        \begin{coxeter}
            \node{1}{(0, 0)}{} \node{2}{(1, 0)}{}
            \node{3}{(2, 0)}{} \node{4}{(3, 0)}{}
            \line{1}{2}{} \line{2}{3}{} \llline{3}{4}
        \end{coxeter} \\ [0pt]
    \end{tabular}
    \caption{Connected elliptic Coxeter diagrams}
    \label{table: connected elliptic coxeter diagrams}
\end{table}

\begin{table}
    \centering
    \begin{tabular}{r l r l}
        \(\widetilde A_1\) &
        \begin{coxeter}
            \node{1}{(0, 0)}{} \node{2}{(1, 0)}{}
            \bline{1}{2}{}
        \end{coxeter} &
        \(\widetilde A_n\)\small\((n \geqslant 2)\) &
        \begin{coxeter}
            \node{1}{(0, -.5)}{} \node{2}{(0, .5)}{}
            \hnode{3}{(.5, 1)} \hnode{4}{(1.5, 1)}
            \node{5}{(2, .5)}{} \node{6}{(2, -.5)}{}
            \hnode{7}{(1.5, -1)} \hnode{8}{(.5, -1)}
            \line{8}{1}{} \line{1}{2}{} \line{2}{3}{}
            \line{4}{5}{} \line{5}{6}{} \line{6}{7}{}
            \dots{3}{4} \dots{7}{8}
        \end{coxeter} \\ [5pt]
    
        \(\widetilde B_n\)\small\((n \geqslant 3)\) &
        \begin{coxeter}
            \node{1}{(0, .5)}{} \node{2}{(0, -.5)}{}
            \node{3}{(1, 0)}{}
            \hnode{4}{(2, 0)} \hnode{5}{(3, 0)}
            \node{6}{(4, 0)}{} \node{7}{(5, 0)}{}
            \line{1}{3}{} \line{2}{3}{} \line{3}{4}{}
            \line{5}{6}{} \lline{6}{7}
            \dots{4}{5}
        \end{coxeter} &
        \(\widetilde C_n\)\small\((n \geqslant 2)\) &
        \begin{coxeter}
            \node{1}{(0, 0)}{} \node{2}{(1, 0)}{}
            \hnode{3}{(2, 0)} \hnode{4}{(3, 0)}
            \node{5}{(4, 0)}{} \node{6}{(5, 0)}{}
            \lline{1}{2} \line{2}{3}{}
            \line{4}{5}{} \lline{5}{6}
            \dots{3}{4}
        \end{coxeter} \\ [10pt]
    
        \(\widetilde D_n\)\small\((n \geqslant 4)\) &
        \begin{coxeter}
            \node{1}{(0, .5)}{} \node{2}{(0, -.5)}{}
            \node{3}{(1, 0)}{}
            \hnode{4}{(2, 0)} \hnode{5}{(3, 0)}
            \node{6}{(4, 0)}{} \node{7}{(5, .5)}{}
            \node{8}{(5, -.5)}{}
            \line{1}{3}{} \line{2}{3}{} \line{3}{4}{}
            \line{5}{6}{} \line{6}{7}{} \line{6}{8}{}
            \dots{4}{5}
        \end{coxeter} &
        \(\widetilde G_2\) &
        \begin{coxeter}
            \node{1}{(0, 0)}{} \node{2}{(1, 0)}{}
            \node{3}{(2, 0)}{}
            \line{1}{2}{} \lllline{2}{3}
        \end{coxeter} \\
    
        \(\widetilde F_4\) &
        \begin{coxeter}
            \node{1}{(0, 0)}{} \node{2}{(1, 0)}{}
            \node{3}{(2, 0)}{} \node{4}{(3, 0)}{}
            \node{5}{(4, 0)}{}
            \line{1}{2}{} \lline{2}{3} \line{3}{4}{}
            \line{4}{5}{}
        \end{coxeter} &
        \(\widetilde E_6\) &
        \begin{coxeter}
            \node{1}{(0, 1)}{} \node{2}{(1, 1)}{}
            \node{3}{(2, 1)}{} \node{4}{(2, 0)}{}
            \node{5}{(2, -1)}{} \node{6}{(3, 1)}{}
            \node{7}{(4, 1)}{}
            \line{1}{2}{} \line{2}{3}{} \line{3}{4}{}
            \line{4}{5}{} \line{3}{6}{} \line{6}{7}{}
        \end{coxeter} \\ [15pt]
    
        \(\widetilde E_7\) &
        \begin{coxeter}
            \node{1}{(0, .5)}{} \node{2}{(1, .5)}{}
            \node{3}{(2, .5)}{} \node{4}{(3, .5)}{}
            \node{5}{(3, -.5)}{} \node{6}{(4, .5)}{}
            \node{7}{(5, .5)}{} \node{8}{(6, .5)}{}
            \line{1}{2}{} \line{2}{3}{} \line{3}{4}{}
            \line{4}{5}{} \line{4}{6}{} \line{6}{7}{}
            \line{7}{8}{}
        \end{coxeter} &
        \(\widetilde E_8\) &
        \begin{coxeter}
            \node{1}{(0, .5)}{} \node{2}{(1, .5)}{}
            \node{3}{(2, .5)}{} \node{4}{(2, -.5)}{}
            \node{5}{(3, .5)}{} \node{6}{(4, .5)}{}
            \node{7}{(5, .5)}{} \node{8}{(6, .5)}{}
            \node{9}{(7, .5)}{}
            \line{1}{2}{} \line{2}{3}{} \line{3}{4}{}
            \line{3}{5}{} \line{5}{6}{} \line{6}{7}{}
            \line{7}{8}{} \line{8}{9}{}
        \end{coxeter} \\ [10pt]
    \end{tabular}
    \caption{Connected parabolic Coxeter diagrams}
    \label{table: connected parabolic coxeter diagrams}
\end{table}

\begin{corollary}
    Every elliptic diagram contains no cycle. Every vertex of an elliptic diagram is joined with at most three other vertices.
\end{corollary}

\begin{table}
	\def\s{1.3}
    \centering
    \begin{tabular}{c | c}
        \small Order & \small Diagrams \\
        \hline \\ [-1.5ex]
        2 &
        \begin{coxeter}[\s]
            \node{1}{(0, -.25)}{} \node{2}{(1, -.25)}{}
            \dline{2}{1}{$\rho$}
        \end{coxeter} \quad \(\rho > 1\) \\ [2ex]
        \hline \\ [-1.5ex]
    
        3 &
        \begin{coxeter}[\s]
            \node{1}{({cos(90) / sqrt(3)}, {sin(90) / sqrt(3)})}{}
            \node{2}{({cos(210) / sqrt(3)}, {sin(210) / sqrt(3)})}{}
            \node{3}{({cos(330) / sqrt(3)}, {sin(330) / sqrt(3)})}{}
            \line{1}{2}{$k$} \line{2}{3}{$m$} \line{3}{1}{$l$}
        \end{coxeter}
        \begin{tabular}{c} \\ [-3.5ex]
            \scriptsize\(\big(2 \leqslant k, l, m < \infty,\) \\
            \scriptsize\(\frac{1}{k} + \frac{1}{l} + \frac{1}{m} < 1\big)\)
        \end{tabular} \\ [.5ex]
        \hline \\ [-1.5ex]
    
        4 &
        \begin{tabular}{c}
            \begin{coxeter}[\s]
                \node{1}{(0, 0)}{} \node{2}{(1, 0)}{}
                \node{3}{(2, 0)}{} \node{4}{(3, 0)}{}
                \line{1}{2}{} \llline{2}{3} \line{3}{4}{}
            \end{coxeter}
            \begin{coxeter}[\s]
                \node{1}{(0, 0)}{} \node{2}{(1, 0)}{}
                \node{3}{(2, .5)}{} \node{4}{(2, -.5)}{}
                \llline{1}{2} \line{2}{3}{} \line{2}{4}{}
            \end{coxeter} \\ [5pt]
    
            \begin{coxeter}[\s]
                \node{1}{(0, 0)}{} \node{2}{(1, 0)}{}
                \node{3}{(2, 0)}{} \node{4}{(3, 0)}{}
                \llline{1}{2} \line{2}{3}{} \lline{3}{4}
            \end{coxeter}
            \begin{coxeter}[\s]
                \node{1}{(0, 0)}{} \node{2}{(1, 0)}{}
                \node{3}{(2, 0)}{} \node{4}{(3, 0)}{}
                \llline{1}{2} \line{2}{3}{} \llline{3}{4}
            \end{coxeter} \\ [5pt]
    
            \begin{coxeter}[\s]
                \node{1}{(0, .5)}{} \node{2}{(1, .5)}{}
                \node{3}{(1, -.5)}{} \node{4}{(0, -.5)}{}
                \lline{1}{2} \line{2}{3}{}
                \line{3}{4}{} \line{4}{1}{}
            \end{coxeter}
            \begin{coxeter}[\s]
                \node{1}{(0, .5)}{} \node{2}{(1, .5)}{}
                \node{3}{(1, -.5)}{} \node{4}{(0, -.5)}{}
                \lline{1}{2} \line{2}{3}{}
                \lline{3}{4} \line{4}{1}{}
            \end{coxeter}
            \begin{coxeter}[\s]
                \node{1}{(0, .5)}{} \node{2}{(1, .5)}{}
                \node{3}{(1, -.5)}{} \node{4}{(0, -.5)}{}
                \llline{1}{2} \line{2}{3}{}
                \line{3}{4}{} \line{4}{1}{}
            \end{coxeter}
            \begin{coxeter}[\s]
                \node{1}{(0, .5)}{} \node{2}{(1, .5)}{}
                \node{3}{(1, -.5)}{} \node{4}{(0, -.5)}{}
                \llline{1}{2} \line{2}{3}{}
                \lline{3}{4} \line{4}{1}{}
            \end{coxeter}
            \begin{coxeter}[\s]
                \node{1}{(0, .5)}{} \node{2}{(1, .5)}{}
                \node{3}{(1, -.5)}{} \node{4}{(0, -.5)}{}
                \llline{1}{2} \line{2}{3}{}
                \llline{3}{4} \line{4}{1}{}
            \end{coxeter} \\ [15pt]
        \end{tabular} \\ [.5ex]
        \hline \\ [-1.5ex]
    
        5 &
        \begin{tabular}{c}
            \begin{coxeter}[\s]
                \node{1}{(0, 0)}{} \node{2}{(1, 0)}{}
                \node{3}{(2, 0)}{} \node{4}{(3, 0)}{}
                \node{5}{(4, 0)}{}
                \llline{1}{2} \line{2}{3}{}
                \line{3}{4}{} \line{4}{5}{}
            \end{coxeter} \\
            \begin{coxeter}[\s]
                \node{1}{(0, 0)}{} \node{2}{(1, 0)}{}
                \node{3}{(2, 0)}{} \node{4}{(3, 0)}{}
                \node{5}{(4, 0)}{}
                \llline{1}{2} \line{2}{3}{}
                \line{3}{4}{} \lline{4}{5}
            \end{coxeter} \\
            \begin{coxeter}[\s]
                \node{1}{(0, 0)}{} \node{2}{(1, 0)}{}
                \node{3}{(2, 0)}{} \node{4}{(3, 0)}{}
                \node{5}{(4, 0)}{}
                \llline{1}{2} \line{2}{3}{}
                \line{3}{4}{} \llline{4}{5}
            \end{coxeter}
        \end{tabular}
    
        \begin{coxeter}[\s]
            \node{1}{(0, 0)}{} \node{2}{(1, 0)}{}
            \node{3}{(2, 0)}{} \node{4}{(3, .5)}{}
            \node{5}{(3, -.5)}{}
            \llline{1}{2} \line{2}{3}{}
            \line{3}{4}{} \line{3}{5}{}
        \end{coxeter}
    
        \begin{coxeter}[\s]
            \node{1}{({cos(90)}, {sin(90)})}{}
            \node{2}{({cos(162)}, {sin(162)})}{}
            \node{3}{({cos(234)}, {sin(234)})}{}
            \node{4}{({cos(306)}, {sin(306)})}{}
            \node{5}{({cos(18)}, {sin(18)})}{}
            \line{1}{2}{} \line{2}{3}{} \lline{3}{4}
            \line{4}{5}{} \line{5}{1}{}
        \end{coxeter} \\ [20pt]
    \end{tabular}
    \caption{Lann\'er diagrams}
    \label{table: lanner diagrams}
\end{table}

A hyperbolic Coxeter diagram \(S\) is called a \emph{Lann\'er diagram} if any proper subdiagram of \(S\) is elliptic. All Lann\'er diagrams were classified by Lann\'er in \cite{L50}. They are listed in Table \ref{table: lanner diagrams}. These diagrams correspond (in the sense defined further) to compact hyperbolic Coxeter simplices. Nevertheless, the importance of such diagrams can already be appreciated.

\begin{proposition} \label{proposition: hyperbolic diagrams}
	Every hyperbolic diagram contains either a parabolic or a Lann\'er subdiagram.
\end{proposition}

\subsection{Hyperbolic Coxeter polytopes}

Let \(P \subset \hyperbolic^d\) be a Coxeter polytope with facets \(f_1, \dots, f_n\). The Coxeter diagram of the polytope \(P\) is a Coxeter diagram with vertices \(v_1, \dots, v_n\). If the facets \(f_i\) and \(f_j\) intersect, then the weight of the edge \(v_i v_j\) is equal to the cosine of the dihedral angle between the facets. If the facets \(f_i\) and \(f_j\) are parallel, then the weight of the edge \(v_i v_j\) is equal to one. If the facets \(f_i\) and \(f_j\) diverge, then the weight of the edge \(v_i v_j\) is equal to the hyperbolic cosine of the distance between \(f_i\) and \(f_j\).

Now let us list the essential results on combinatorics of compact hyperbolic Coxeter polytopes. Let \(P\) be a polytope. By \(\faces(P)\) we denote the partially ordered set of its faces. Let \(S\) be a Coxeter diagram. By \(\faces(S)\) we denote the dual (i.e., anti-isomorphic to the original) partially ordered set of its elliptic subdiagrams.

\begin{proposition}[{\cite[Theorem 3.1]{V85}}]
    Let \(P \subset \hyperbolic^d\) be a compact hyperbolic Coxeter polytope. Partially ordered sets \(\faces(S(P))\) and \(\faces(P)\) are isomorphic.
\end{proposition}

Thus, the combinatorics of a compact polytope can be easily read according to its Coxeter diagram. A set of facets has a non-empty intersection if and only if the subdiagram generated by the corresponding vertices is elliptic.

Consider a compact hyperbolic Coxeter polytope. The structure of its Coxeter diagram is restricted by the propositions below.

\begin{proposition}[{\cite[Proposition 3.2]{V85}}]
    Let \(P \subset \hyperbolic^d\) be a compact hyperbolic Coxeter polytope. The Coxeter diagram \(S(P)\) contains no parabolic subdiagrams.
\end{proposition}

\begin{proposition}[{\cite[Proposition 4.2]{V85}}] \label{proposition: combinatorial structure of coxeter polytopes}
    A Coxeter diagram \(S\) is a Coxeter diagram of a compact hyperbolic Coxeter polytope if and only if the diagram is hyperbolic, contains no parabolic subdiagrams, and there is a polytope \(P \subset \euclidean^d\) such that \(\faces(P)\) and \(\faces(S)\) are isomorphic.
\end{proposition}

The following statement is an easy corollary of the propositions above.

\begin{corollary}
    A polytope \(P \subset \hyperbolic^d\) is a compact simplex if and only if \(S(P)\) is a Lann\'er diagram.
\end{corollary}

Finally, the best known general estimation on dimension of a compact hyperbolic Coxeter polytope is the following.

\begin{theorem}[{\cite[Theorem 1]{V84}}]
    There are no compact Coxeter polytopes in~ \(\hyperbolic^{\geqslant 30}\).
\end{theorem}

\subsection{Superhyperbolic diagrams}

A Coxeter diagram is said to be \emph{superhyperbolic} if its negative inertia index is greater than one. A \emph{local determinant} of a diagram \(S\) on its subdiagram \(T\) is \[\det(S, T) = \frac{\det(S)}{\det(S \setminus T)}.\] Usually we will mark the vertices of the subdiagram \(T\) with \(\pmb\vee\).

We denote by \(p(\gamma)\) the product of the edge weights of a cycle \(\gamma\). The following proposition is very useful for computing  determinants.

\begin{proposition}[{\cite[Proposition 11]{V84}}]
    A determinant of a Coxeter diagram~\(S\) is equal to the sum of the products \[(-1)^k \cdot p(\gamma_1) \cdot \ldots \cdot p(\gamma_k)\] over all sets \(\{\gamma_1, \dots, \gamma_k\}\) of positive length disjoint cycles.
\end{proposition}

\begin{proposition}[{\cite[Proposition 13]{V84}}] \label{proposition: local determinant of two diagrams joined by elliptic edge}
    If a Coxeter diagram \(S\) is generated by two disjoint subdiagrams \(S_1\) and \(S_2\) joined by a unique edge \(v_1 v_2\) of weight~\(w\), then \[\det(S, \langle v_1, v_2 \rangle) = \det(S_1, v_1) \cdot \det(S_2, v_2) - w^2.\]
\end{proposition}

\begin{proposition}[{\cite[Table 2]{V84}}]
    \[\det\left(\begin{coxeter}
        \node{1}{({cos(120)}, {sin(120)})}{}
        \node{2}{({cos(240)}, {sin(240)})}{}
        \node{3}{(1, 0)}{$\pmb\vee$}
        \line{1}{2}{$m$}\line{2}{3}{$l$}\line{3}{1}{$k$}
    \end{coxeter}\right) = -d(k, l, m),\] where \[d(k, l, m) = \frac{\cos\left(\frac{\pi}{k}\right)^2 + \cos\left(\frac{\pi}{l}\right)^2 + 2 \cos\left(\frac{\pi}{k}\right) \cos\left(\frac{\pi}{l}\right) \cos\left(\frac{\pi}{m}\right)}{\sin\left(\frac{\pi}{m}\right)^2} - 1.\]
\end{proposition}

\medskip

Now let us use these propositions to test the diagram below for hyperbolicity.
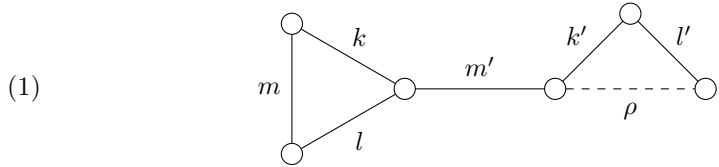
\begin{equation}\label{equation: lanner 3-diagram, 2-diagram, and a vertex}\begin{coxeter}[2]
    \node{1}{({cos(120)}, {sin(120)})}{}
    \node{2}{({cos(240)}, {sin(240)})}{}
    \node{3}{(1, 0)}{}\node{4}{(3, 0)}{}
    \node{5}{(4, 1)}{}\node{6}{(5, 0)}{}
    \line{1}{2}{$m$}\line{2}{3}{$l$}\line{3}{1}{$k$}
    \line{4}{3}{$m'$}\line{5}{4}{$k'$}\line{6}{5}{$l'$}
    \dline{4}{6}{$\rho$}
\end{coxeter}\end{equation}
This diagram contains an elliptic subdiagram of order \(4\) and a Lann\'er subdiagram of order \(2\). Therefore, its signature is either \((4, 1, 1)\) or \((5, 1, 0)\), or \((4, 2, 0)\). Hence, the diagram is hyperbolic if and only if
\[
	\det\left(\begin{coxeter}
    	\node{1}{({cos(120)}, {sin(120)})}{}
	    \node{2}{({cos(240)}, {sin(240)})}{}
	    \node{3}{(1, 0)}{$\pmb\vee$}\node{4}{(3, 0)}{$\pmb\vee$}
	    \node{5}{(4, 1)}{}\node{6}{(5, 0)}{}
	    \line{1}{2}{$m$}\line{2}{3}{$l$}\line{3}{1}{$k$}
    	\line{4}{3}{$m'$}\line[0.3]{5}{4}{$k'$}\line{6}{5}{$l'$}
	    \dline{4}{6}{$\rho$}
	\end{coxeter}\right) \leqslant 0.
\]
But
\begin{multline*}
    \det\left(\begin{coxeter}
        \node{1}{({cos(120)}, {sin(120)})}{}
        \node{2}{({cos(240)}, {sin(240)})}{}
        \node{3}{(1, 0)}{$\pmb\vee$}\node{4}{(3, 0)}{$\pmb\vee$}
        \node{5}{(4, 1)}{}\node{6}{(5, 0)}{}
        \line{1}{2}{$m$}\line{2}{3}{$l$}\line{3}{1}{$k$}
        \line{4}{3}{$m'$}\line[0.3]{5}{4}{$k'$}\line{6}{5}{$l'$}
        \dline{4}{6}{$\rho$}
    \end{coxeter}\right) = \det\left(\begin{coxeter}
        \node{1}{({cos(120)}, {sin(120)})}{}
        \node{2}{({cos(240)}, {sin(240)})}{}
        \node{3}{(1, 0)}{$\pmb\vee$}
        \line{1}{2}{$m$}\line{2}{3}{$l$}\line{3}{1}{$k$}
    \end{coxeter}\right) \cdot \, \det\left(\begin{coxeter}
        \node{4}{(3, 0)}{$\pmb\vee$}
        \node{5}{(4, 1)}{}\node{6}{(5, 0)}{}
        \line[0.3]{5}{4}{$k'$}\line{6}{5}{$l'$}
        \dline{4}{6}{$\rho$}
    \end{coxeter}\right) - \cos\left(\frac{\pi}{m'}\right)^2
    \\
    = d(k, l, m) \frac{\cos\left(\frac{\pi}{l'}\right)^2 + \cos\left(\frac{\pi}{k'}\right)^2 + \rho^2 + 2 \rho\cos\left(\frac{\pi}{k'}\right) \cos\left(\frac{\pi}{l'}\right) - 1}{\sin\left(\frac{\pi}{l'}\right)^2} - \cos\left(\frac{\pi}{m'}\right)^2.
\end{multline*}
If \(d(k, l, m) \ne 0\), then the last inequality is equivalent to the following:
\[
	\rho^2 + 2 \rho\cos\left(\frac{\pi}{k'}\right) \cos\left(\frac{\pi}{l'}\right) + \cos\left(\frac{\pi}{l'}\right)^2 + \cos\left(\frac{\pi}{k'}\right)^2 - 1 - \frac{\sin\left(\frac{\pi}{l'}\right)^2 \cos\left(\frac{\pi}{m'}\right)^2}{d(k, l, m)} \leqslant 0.
\]
Consider the left part of this inequality as a quadratic function in \(\rho\). One of the zeros of this function is not greater than \(1\). So there is a \(\rho > 1\) satisfying the inequality if and only if for \(\rho = 1\) the strict inequality holds, i.e.
\[
	D(k, l, m, k', l', m') = \left(\cos\left(\frac{\pi}{l'}\right) + \cos\left(\frac{\pi}{k'}\right)\right)^2 - \frac{\sin\left(\frac{\pi}{l'}\right)^2 \cos\left(\frac{\pi}{m'}\right)^2}{d(k, l, m)} < 0.
\]
This proves the following lemma.

\begin{lemma} \label{lemma: superhyperbolic diagram from lanner 3-diagram, 2-diagram, and a vertex}
    Let \(\begin{coxeter}
        \node{1}{({cos(120)}, {sin(120)})}{}
        \node{2}{({cos(240)}, {sin(240)})}{}
        \node{3}{(1, 0)}{}
        \line{1}{2}{$m$}\line{2}{3}{$l$}\line{3}{1}{$k$}
    \end{coxeter}\) be a Lann\'er diagram. The Coxeter diagram (\ref{equation: lanner 3-diagram, 2-diagram, and a vertex}) is superhyperbolic for any \(\rho > 1\) if and only if 
    \[
    	D(k, l, m, k', l', m') \geqslant 0.
    \]
\end{lemma}

\begin{remark} \label{remark: monotonicity}
	Direct calculations show that if \(d(k, l, m) > 0\), then the function \(D\) is increasing in \(k, l, m, k', l'\), and decreasing in \(m'\).
\end{remark}

\section{Proof of Theorems~A and B} \label{section: product of simplices}

Let \(\Sigma_1\) and \(\Sigma_2\) be sets of Coxeter diagram. By \(\Sigma_1 \times_k \Sigma_2\) we denote the set of all Coxeter diagrams \(S\) generated by subdiagrams \(S_1 \in \Sigma_1\) and \(S_2 \in \Sigma_2\) such that intersection \(S_1 \cap S_2\) consists of \(k\) vertices and every Lann\'er or parabolic subdiagram is contained in either \(S_1\) or \(S_2\).

Denote by \(\mathcal{L}_k\) the set of all Lann\'er diagrams of order \(k\) and by \(\simplex_k\) the standard \((k - 1)\)-dimensional simplex. Consider a compact hyperbolic Coxeter polytope~\(P\). Suppose that \(\faces(P)\) and \(\faces(\simplex_{k_1} \times \dots \times \simplex_{k_n})\) are isomorphic. Every face of \(\simplex_{k_1} \times \dots \times \simplex_{k_n}\) is equal to \(f_1 \times \dots \times f_n\) for some faces \(f_i\) of \(\simplex_i\). Therefore, the facets of \(\simplex_{k_1} \times \dots \times \simplex_{k_n}\) are equal to 
\[
	f^i_j = \simplex_{k_1} \times \dots \times \simplex_{k_{i - 1}} \times f_j \times \simplex_{k_{i + 1}} \times \dots \times \simplex_{k_n}, \quad \text{where } f_j \text{ is a facet of } \simplex_{k_i}.
\] 
Let \(F\) be a set of the facets. The intersection \(\bigcap_{f \in F} f\) is empty if and only if \(\{f^i_1, \dots, f^i_{k_i}\} \subseteq F\) for some \(1 \leqslant i \leqslant n\). According to Proposition~\ref{proposition: hyperbolic diagrams} and Proposition~ \ref{proposition: combinatorial structure of coxeter polytopes}, \(S(P) \in \mathcal{L}_{k_1} \times_0 \dots \times_0 \mathcal{L}_{k_n}\). 

Without loss of generality, \(k_1 \geqslant \dots \geqslant k_n\). If \(k_1 = \dots = k_n = 2\), then \(P\) is a \(n\)-dimensional cube. If \(k_n \ne 2\), then the diagram \(S(P)\) contains no dashed edges. It is known that every such polytope is a product of at most two simplices (see \cite[Theorem~A]{FT08}). Thus, Theorem~\ref{theorem: products of simplices} is a corollary of Theorem~\ref{theorem: products of lanner diagrams}. For the reader's convenience we present its statement again. The proof starts below.

\TheoremA*

\subsection{Case \texorpdfstring{\(k_1 \geqslant 4\)}{k\_1 >= 4}}

Let \(S = \langle L_1, \dots, L_n \rangle\) be a Coxeter diagram generated by disjoined Lann\'er diagrams \(L_1, \dots, L_n\) of orders \(k_1, \dots, k_n\). Let \(v_2, \dots, v_n\) be arbitrary vertices of the subdiagrams \(L_2, \dots, L_n\) respectively. The diagram \(\langle L_1, v_2, v_3, v_4 \rangle\) does not contain any parabolic or Lann\'er subdiagrams other than \(L_1\). But direct calculations show that each such diagram contains either a parabolic or a new Lann\'er subdiagram (see \cite[Lemma 4.2]{E96}). In this case, we say that no Lann\'er diagram of order \(4\) or \(5\) can be \emph{expanded} with three vertices without forming a new Lann\'er or parabolic subdiagram. In other words, the sets
\(\mathcal{L}_4 \times_0
\{\begin{coxeter}\node{1}{(0, 0)}{}\end{coxeter}\} \times_0
\{\begin{coxeter}\node{1}{(0, 0)}{}\end{coxeter}\} \times_0
\{\begin{coxeter}\node{1}{(0, 0)}{}\end{coxeter}\}\)
and
\(\mathcal{L}_5 \times_0
\{\begin{coxeter}\node{1}{(0, 0)}{}\end{coxeter}\} \times_0
\{\begin{coxeter}\node{1}{(0, 0)}{}\end{coxeter}\} \times_0
\{\begin{coxeter}\node{1}{(0, 0)}{}\end{coxeter}\}\)
are empty (\(\begin{coxeter}\node{1}{(0, 0)}{}\end{coxeter}\) is a Coxeter diagram consisting of one vertex). Therefore, the required diagrams with \(k_1 \geqslant 4\) do not exist.

\subsection{Case \texorpdfstring{\(k_1 = k_2 = 3\)}{k\_1 = k\_2 = 3}}

Consider the Lann\'er subdiagrams \(L_1 = \langle u_1, u_2, u_3 \rangle\) and \(L_2 = \langle u_4, u_5, u_6 \rangle\) of order~\(3\). It is shown in \cite[Lemma 4.10]{T07} that if \(|\det(L_1, u_3)| \leqslant |\det(L_2, u_4)|\), then the subdiagram \(\langle L_1, u_4 \rangle\) is one of the following.

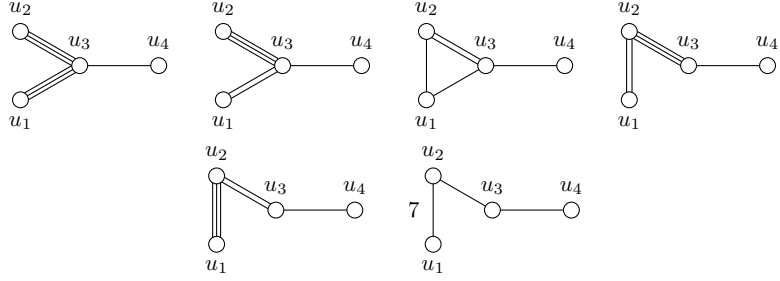
\begin{center}
    \begin{coxeter}[1.5]
        \def\l{.7}
        \node[-90]{1}{({\l*cos(240)}, {\l*sin(240)})}{$u_1$}
        \node{2}{({\l*cos(120)}, {\l*sin(120)})}{$u_2$}
        \node{3}{(\l, 0)}{$u_3$}\node{4}{(3*\l, 0)}{$u_4$}
        \llline{1}{3}{}\llline{2}{3}{}\line{3}{4}{}
    \end{coxeter}
    \begin{coxeter}[1.5]
        \def\l{.7}
        \node[-90]{1}{({\l*cos(240)}, {\l*sin(240)})}{$u_1$}
        \node{2}{({\l*cos(120)}, {\l*sin(120)})}{$u_2$}
        \node{3}{(\l, 0)}{$u_3$}\node{4}{(3*\l, 0)}{$u_4$}
        \lline{1}{3}{}\llline{2}{3}{}\line{3}{4}{}
    \end{coxeter}
    \begin{coxeter}[1.5]
        \def\l{.7}
        \node[-90]{1}{({\l*cos(240)}, {\l*sin(240)})}{$u_1$}
        \node{2}{({\l*cos(120)}, {\l*sin(120)})}{$u_2$}
        \node{3}{(\l, 0)}{$u_3$}\node{4}{(3*\l, 0)}{$u_4$}
        \line{1}{3}{}\lline{2}{3}{}\line{3}{4}{}\line{1}{2}{}
    \end{coxeter}
    \begin{coxeter}[1.5]
        \def\l{.7}
        \node[-90]{1}{({\l*cos(240)}, {\l*sin(240)})}{$u_1$}
        \node{2}{({\l*cos(120)}, {\l*sin(120)})}{$u_2$}
        \node{3}{(\l, 0)}{$u_3$}\node{4}{(3*\l, 0)}{$u_4$}
        \lline{1}{2}{}\llline{2}{3}{}\line{3}{4}{}
    \end{coxeter}
    \begin{coxeter}[1.5]
        \def\l{.7}
        \node[-90]{1}{({\l*cos(240)}, {\l*sin(240)})}{$u_1$}
        \node{2}{({\l*cos(120)}, {\l*sin(120)})}{$u_2$}
        \node{3}{(\l, 0)}{$u_3$}\node{4}{(3*\l, 0)}{$u_4$}
        \llline{1}{2}{}\lline{2}{3}{}\line{3}{4}{}
    \end{coxeter}
    \begin{coxeter}[1.5]
        \def\l{.7}
        \node[-90]{1}{({\l*cos(240)}, {\l*sin(240)})}{$u_1$}
        \node{2}{({\l*cos(120)}, {\l*sin(120)})}{$u_2$}
        \node{3}{(\l, 0)}{$u_3$}\node{4}{(3*\l, 0)}{$u_4$}
        \line{2}{1}{7}\line{2}{3}{}\line{3}{4}{}
    \end{coxeter}
\end{center}
The diagram \(L_1\) can be expanded with two vertices, so the subdiagram \(\langle L_1, u_4 \rangle\) is the following, or a new Lann\'er or parabolic subdiagram is forming.

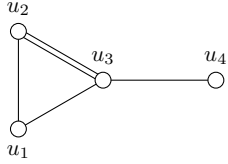
\begin{center}\begin{coxeter}[1.5]
    \node[-90]{1}{({cos(240)}, {sin(240)})}{$u_1$}
    \node{2}{({cos(120)}, {sin(120)})}{$u_2$}
    \node{3}{(1, 0)}{$u_3$}\node{4}{(3, 0)}{$u_4$}
    \line{1}{3}{}\lline{2}{3}{}\line{3}{4}{}\line{1}{2}{}
\end{coxeter}\end{center}

So, \(|\det(L_1, u_3)| = \frac{\sqrt{2}}{3}\). According to Proposition \ref{proposition: local determinant of two diagrams joined by elliptic edge}, \(|\det(L_2, u_4)| \leqslant \frac{3}{4\sqrt{2}}\). The multiplicity of the edges \(u_4 u_5\) and \(u_4 u_6\) does not exceed one. There is the only Lann\'er diagram of order \(3\) with such properties, which is shown below.

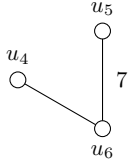
\begin{center}\begin{coxeter}[1.5]
    \node{4}{(-1, 0)}{$u_4$}
    \node{5}{({cos(60)}, {sin(60)})}{$u_5$}
    \node[-90]{6}{({cos(-60)}, {sin(-60)})}{$u_6$}
    \line{4}{6}{}\line{6}{5}{7}
\end{coxeter}\end{center}
This diagram is not appropriate since it cannot be expanded with three vertices without forming a new Lann\'er or parabolic subdiagram.

\subsection{Case \texorpdfstring{\(k_2 = 2\)}{k\_2 = 2}}

A Lann\'er diagram of order \(3\) cannot be expanded with five vertices. Therefore, \(n \leqslant 5\). Let us denote by \([u, v]\) the multiplicity of the edge connecting vertices \(u\) and \(v\).

\begin{lemma}
    Under the conditions described above, \(n \leqslant 4\).
\end{lemma}

\begin{proof}
    Suppose that \(n = 5\). Denote the Lann\'er subdiagrams by \(L_1 = \langle u_1, u_2, u_3 \rangle\), \(L_2 = \langle u_4, u_8 \rangle\), \(L_3 = \langle u_5, u_9 \rangle\), \(L_4 = \langle u_6, u_{10} \rangle\), and \(L_5 = \langle u_7, u_{11} \rangle\). Without loss~of generality, the vertices \(u_4\), \(u_5\), \(u_6\), and \(u_7\) are joined to the subdiagram~\(L_1\). The only subdiagram \(\langle L_1, u_4, u_5, u_6, u_7 \rangle\) that satisfies these properties is shown below.

    \begin{center}\begin{coxeter}[1.5]
        \node[180]{1}{(0, 1)}{$u_1$}
        \node[170]{2}{({cos(210)}, {sin(210)})}{$u_2$}
        \node[10]{3}{({cos(330)}, {sin(330)})}{$u_3$}
        \node{4}{({cos(120)}, {sin(120) + 1})}{$u_4$}
        \node{5}{({cos(60)}, {sin(60) + 1})}{$u_5$}
        \node[-90]{6}{({cos(210) + cos(-120)}, {sin(210) + sin(-120)})}{$u_6$}
        \node[-90]{7}{({cos(330) + cos(-60)}, {sin(330) + sin(-60)})}{$u_7$}
        \line{1}{2}{}\line{1}{3}{}\lline{2}{3}
        \line{1}{4}{}\line{1}{5}{}
        \line{2}{6}{}\line{3}{7}{}
    \end{coxeter}\end{center}

    It is easy to check that 
    \begin{align*}
    	[u_6, u_4] = [u_6, u_5] = [u_6, u_8] = [u_6, u_9] &= \\ 
    	[u_7, u_4] = [u_7, u_5] = [u_7, u_8] = [u_7, u_9] &=  0.
    \end{align*} 
    This implies that the vertices \(u_{10}\) and \(u_{11}\) are joined to the subdiagrams \(\langle u_4, u_8 \rangle\) and \(\langle u_5, u_9 \rangle\). There are two cases: \begin{enumerate}
        \item Let \([u_{10}, u_{11}] \geqslant 1\). Without loss of generality, we may assume that 
        \[
        	[u_{10}, u_8] = [u_{10}, u_5] = [u_{11}, u_4] = [u_{11}, u_9] \geqslant 1
        \] 
        and 
        \[
        	[u_{11}, u_8] = [u_{11}, u_5] = [u_{10}, u_4] = [u_{10}, u_9] = 0.
        \] 
        Then 
        \[
        	[u_4, u_5] = [u_4, u_9] = [u_8, u_5] = [u_8, u_9] = 0
        \] 
        and the subdiagram \(\langle L_2, L_3 \rangle\) is not connected.
        
        \item Let \([u_{10}, u_{11}] = 0\). Then, without loss of generality, we may assume that \([u_6, u_{11}] = 1\). In this case the subdiagram \(L_5\) can be joined with \(L_2\) and \(L_3\) only if 
        \[
        	[u_{11}, u_8] = [u_{11}, u_9] \geqslant 1.
        \] 
        Then 
        \[
        	[u_4, u_5] = [u_4, u_9] = [u_8, u_5] = [u_8, u_9] = 0
        \] 
        and the subdiagram \(\langle L_2, L_3 \rangle\) is not connected.
    \end{enumerate}
\end{proof}

Thus, only the products of a triangle and a \(3\)-dimensional cube left. Denote the Lann\'er subdiagrams by \(L_1 = \langle u_1, u_2, u_3 \rangle\), \(L_2 = \langle u_4, u_7 \rangle\), \(L_3 = \langle u_5, u_8 \rangle\), and \(L_4 = \langle u_6, u_9 \rangle\). We suppose that the subdiagrams \(\langle L_1, u_4 \rangle\), \(\langle L_1, u_5 \rangle\), and \(\langle L_1, u_6 \rangle\) are connected. If the subdiagram \(\langle L_1, u_4, u_5, u_6 \rangle\) contains the only Lann\'er subdiagram, then all edges of the subdiagram \(L_1\) have a positive multiplicity. This means that any vertex of the subdiagrams \(L_2\), \(L_3\), and \(L_4\) is joined to \(L_1\) by at most one edge. Denote the multiplicity of such an edge by \([u, L_1]\). If \([u_7, L_1] \geqslant 1\) and \([u_8, L_1] \geqslant 1\), then \(L_2\) and \(L_3\) are not connected. Thus, without loss of generality, \([u_8, L_1] = [u_9, L_1] = 0\), \([u_4, L_1] \geqslant [u_7, L_1]\), and \([u_5, L_1] \geqslant [u_6, L_1] = 1\).

\begin{lemma}
    If \([u_5, L_1] \geqslant 2\), then \([u_7, L_1] = 0\).
\end{lemma}

\begin{proof}
    Assume that \([u_7, L_1] \geqslant 1\). The only possible subdiagram \(\langle L_1, L_2, u_5, u_6 \rangle\) is shown below.

    \begin{center}\begin{coxeter}[1.5]
        \node[180]{1}{({cos(90)}, {sin(90)})}{$u_1$}
        \node[135]{2}{({cos(210)}, {sin(210)})}{$u_2$}
        \node[45]{3}{({cos(330)}, {sin(330)})}{$u_3$}
        \node{4}{({cos(90) + sqrt(3)*cos(120)}, {sin(90) + sqrt(3)*sin(120)})}{$u_4$}
        \node{7}{({cos(90) + sqrt(3)*cos(60)}, {sin(90) + sqrt(3)*sin(60)})}{$u_7$}
        \node{5}{({(1 + sqrt(3))*cos(210)}, {(1 + sqrt(3))*sin(210)})}{$u_5$}
        \node{6}{({(1 + sqrt(3))*cos(330)}, {(1 + sqrt(3))*sin(330)})}{$u_6$}
        \line{1}{2}{}\line{2}{3}{}\lline{3}{1}
        \line{1}{4}{}\line{1}{7}{}\dline{4}{7}{}
        \line[.6]{5}{2}{$\geqslant 4$}\line{3}{6}{}
    \end{coxeter}\end{center}
    Then \([u_5, u_4] = [u_5, u_7] = [u_8, u_4] = [u_8, u_7] = 0\) and the subdiagrams \(L_2\) and \(L_3\) are not connected.
\end{proof}

We may suppose that \([u_5, L_1] = 1\) since otherwise we can swap \(L_2\) and \(L_3\). The vertex \(u_8\) is joined to \(L_4\) or the vertex \(u_9\) is joined to \(L_3\). Without loss of generality, \(u_9\) is joined to \(L_3\). The only possible diagram \(\langle L_1, L_3, u_9 \rangle\) is shown below, \(k' \geqslant 3\) or \(l' \geqslant 3\).

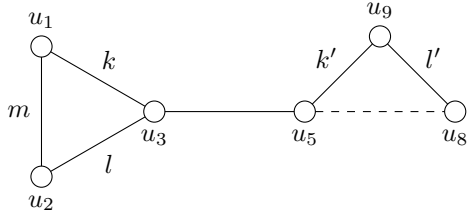
\begin{center}\begin{coxeter}[2]
    \node{1}{({cos(120)}, {sin(120)})}{$u_1$}
    \node[-90]{2}{({cos(240)}, {sin(240)})}{$u_2$}
    \node[-90]{3}{(1, 0)}{$u_3$}
    \node[-90]{5}{(3, 0)}{$u_5$}
    \node{9}{(4, 1)}{$u_9$}
    \node[-90]{8}{(5, 0)}{$u_8$}
    \line{1}{2}{$m$}\line{2}{3}{$l$}\line{3}{1}{$k$}
    \line{5}{3}{}\line{9}{5}{$k'$}\line{8}{9}{$l'$}
    \dline{5}{8}{}
\end{coxeter}\end{center}

Some superhyperbolic diagrams of this form are listed below.

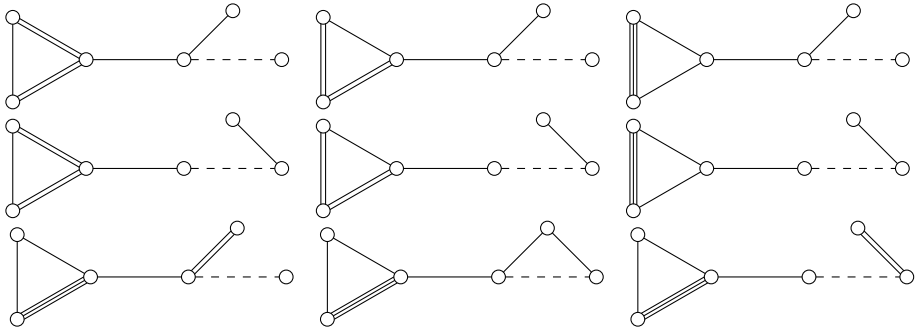
\begin{center}
	\def\s{1.3}
    \begin{coxeter}[\s]
        \node{1}{({cos(120)}, {sin(120)})}{}
        \node{2}{({cos(240)}, {sin(240)})}{}
        \node{3}{(1, 0)}{}\node{4}{(3, 0)}{}
        \node{5}{(4, 1)}{}\node{6}{(5, 0)}{}
        \line{1}{2}{}\lline{2}{3}\lline{3}{1}
        \line{4}{3}{}\line{5}{4}{}
        \dline{4}{6}{}
    \end{coxeter}
    \begin{coxeter}[\s]
        \node{1}{({cos(120)}, {sin(120)})}{}
        \node{2}{({cos(240)}, {sin(240)})}{}
        \node{3}{(1, 0)}{}\node{4}{(3, 0)}{}
        \node{5}{(4, 1)}{}\node{6}{(5, 0)}{}
        \lline{1}{2}{}\lline{2}{3}\line{3}{1}{}
        \line{4}{3}{}\line{5}{4}{}
        \dline{4}{6}{}
    \end{coxeter}
    \begin{coxeter}[\s]
        \node{1}{({cos(120)}, {sin(120)})}{}
        \node{2}{({cos(240)}, {sin(240)})}{}
        \node{3}{(1, 0)}{}\node{4}{(3, 0)}{}
        \node{5}{(4, 1)}{}\node{6}{(5, 0)}{}
        \llline{1}{2}\line{2}{3}{}\line{3}{1}{}
        \line{4}{3}{}\line{5}{4}{}
        \dline{4}{6}{}
    \end{coxeter}
    
    \begin{coxeter}[\s]
        \node{1}{({cos(120)}, {sin(120)})}{}
        \node{2}{({cos(240)}, {sin(240)})}{}
        \node{3}{(1, 0)}{}\node{4}{(3, 0)}{}
        \node{5}{(4, 1)}{}\node{6}{(5, 0)}{}
        \line{1}{2}{}\lline{2}{3}\lline{3}{1}
        \line{4}{3}{}\line{5}{6}{}
        \dline{4}{6}{}
    \end{coxeter}
    \begin{coxeter}[\s]
        \node{1}{({cos(120)}, {sin(120)})}{}
        \node{2}{({cos(240)}, {sin(240)})}{}
        \node{3}{(1, 0)}{}\node{4}{(3, 0)}{}
        \node{5}{(4, 1)}{}\node{6}{(5, 0)}{}
        \lline{1}{2}\lline{2}{3}\line{3}{1}{}
        \line{4}{3}{}\line{5}{6}{}
        \dline{4}{6}{}
    \end{coxeter}
    \begin{coxeter}[\s]
        \node{1}{({cos(120)}, {sin(120)})}{}
        \node{2}{({cos(240)}, {sin(240)})}{}
        \node{3}{(1, 0)}{}\node{4}{(3, 0)}{}
        \node{5}{(4, 1)}{}\node{6}{(5, 0)}{}
        \llline{1}{2}\line{2}{3}{}\line{3}{1}{}
        \line{4}{3}{}\line{5}{6}{}
        \dline{4}{6}{}
    \end{coxeter}

    \begin{coxeter}[\s]
        \node{1}{({cos(120)}, {sin(120)})}{}
        \node{2}{({cos(240)}, {sin(240)})}{}
        \node{3}{(1, 0)}{}\node{4}{(3, 0)}{}
        \node{5}{(4, 1)}{}\node{6}{(5, 0)}{}
        \line{1}{2}{}\llline{2}{3}\line{3}{1}{}
        \line{4}{3}{}\lline{5}{4}
        \dline{4}{6}{}
    \end{coxeter}
    \begin{coxeter}[\s]
        \node{1}{({cos(120)}, {sin(120)})}{}
        \node{2}{({cos(240)}, {sin(240)})}{}
        \node{3}{(1, 0)}{}\node{4}{(3, 0)}{}
        \node{5}{(4, 1)}{}\node{6}{(5, 0)}{}
        \line{1}{2}{}\llline{2}{3}\line{3}{1}{}
        \line{4}{3}{}\line{5}{6}{}\line{4}{5}{}
        \dline{4}{6}{}
    \end{coxeter}
    \begin{coxeter}[\s]
        \node{1}{({cos(120)}, {sin(120)})}{}
        \node{2}{({cos(240)}, {sin(240)})}{}
        \node{3}{(1, 0)}{}\node{4}{(3, 0)}{}
        \node{5}{(4, 1)}{}\node{6}{(5, 0)}{}
        \line{1}{2}{}\llline{2}{3}\line{3}{1}{}
        \line{4}{3}{}\lline{5}{6}
        \dline{4}{6}{}
    \end{coxeter}
\end{center}

We also would like to note that the following diagrams contain a parabolic subdiagram.

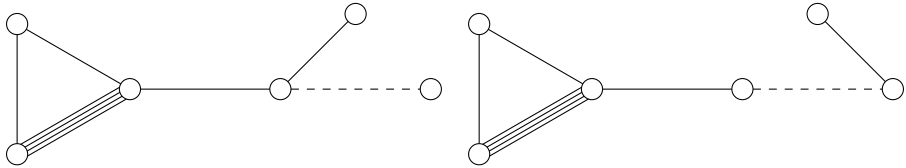
\begin{center}
    \begin{coxeter}[2]
        \node{1}{({cos(120)}, {sin(120)})}{}
        \node{2}{({cos(240)}, {sin(240)})}{}
        \node{3}{(1, 0)}{}\node{4}{(3, 0)}{}
        \node{5}{(4, 1)}{}\node{6}{(5, 0)}{}
        \line{1}{2}{}\lllline{2}{3}\line{3}{1}{}
        \line{4}{3}{}\line{5}{4}{}
        \dline{4}{6}{}
    \end{coxeter}
    \begin{coxeter}[2]
        \node{1}{({cos(120)}, {sin(120)})}{}
        \node{2}{({cos(240)}, {sin(240)})}{}
        \node{3}{(1, 0)}{}\node{4}{(3, 0)}{}
        \node{5}{(4, 1)}{}\node{6}{(5, 0)}{}
        \line{1}{2}{}\lllline{2}{3}\line{3}{1}{}
        \line{4}{3}{}\line{5}{6}{}
        \dline{4}{6}{}
    \end{coxeter}
\end{center}

\begin{lemma}
	The diagram \(L_1\)	is equal to the following diagram.

	\begin{equation}\begin{coxeter}[2] \label{diagram: L2}
	    \node{1}{(0, 1)}{}
    	\node{2}{({cos(210)}, {sin(210)})}{}
	    \node{3}{({cos(330)}, {sin(330)})}{}
    	\line{1}{2}{}\line{1}{3}{}\lline{2}{3}
	\end{coxeter}\end{equation}
\end{lemma}

\begin{proof}
	Combining Lemma \ref{lemma: superhyperbolic diagram from lanner 3-diagram, 2-diagram, and a vertex}, monotonicity of the function \(D\), and a simple computation, we get that either the subdiagram \(L_1\) is equal to (\ref{diagram: L2}), or the subdiagram \(\langle L_1, L_3, u_9 \rangle\) is equal to \(\begin{coxeter}
        \node{1}{({cos(120)}, {sin(120)})}{}
        \node{2}{({cos(240)}, {sin(240)})}{}
        \node{3}{(1, 0)}{}\node{4}{(3, 0)}{}
        \node{5}{(4, 1)}{}\node{6}{(5, 0)}{}
        \line{1}{2}{}\llline{2}{3}\line{3}{1}{}
        \line{4}{3}{}\line{5}{4}{}
        \dline{4}{6}{}
    \end{coxeter}\) or \(\begin{coxeter}
        \node{1}{({cos(120)}, {sin(120)})}{}
        \node{2}{({cos(240)}, {sin(240)})}{}
        \node{3}{(1, 0)}{}\node{4}{(3, 0)}{}
        \node{5}{(4, 1)}{}\node{6}{(5, 0)}{}
        \line{1}{2}{}\llline{2}{3}\line{3}{1}{}
        \line{4}{3}{}\line{5}{6}{}
        \dline{4}{6}{}
    \end{coxeter}\). Without loss of generality, the subdiagram \(\langle L_2, u_9 \rangle\) is connected, i.e. the subdiagram \(\langle L_1, L_2, u_9 \rangle\) is equal to the following diagram, \(k' \geqslant 3\) or \(l' \geqslant 3\).
	\begin{center}\begin{coxeter}[2]
    	\node{2}{({cos(120)}, {sin(120)})}{$u_2$}
	    \node[-90]{3}{({cos(240)}, {sin(240)})}{$u_3$}
    	\node[-90]{1}{(1, 0)}{$u_1$}
	    \node[-90]{4}{(3, 0)}{$u_4$}
    	\node{9}{(4, 1)}{$u_9$}
	    \node[-90]{7}{(5, 0)}{$u_7$}
    	\line{1}{2}{}\llline{2}{3}\line{3}{1}{}
	    \line{4}{1}{}\line{9}{4}{$k'$}\line{7}{9}{$l'$}
    	\dline{4}{7}{}
	\end{coxeter}\end{center}
	The diagram is superhyperbolic.
\end{proof}

\begin{lemma}
	Let \(S\) be a diagram that contains a hyperbolic subdiagram and let \(v \not\in S\) be a~vertex that is joined with the only vertex \(w \not\in S\) by a dotted edge. If the inequality \[\det\big(\langle w, S \rangle\big) - \det(S) > 0\] holds, then the diagram \(\langle v, w, S \rangle\) is superhyperbolic.

    \begin{center}\begin{coxeter}[2]
        \node{v}{(0, 0)}{$v$} \node{w}{(2, 0)}{$w$}
        \dline[.4]{v}{w}{$\rho$}

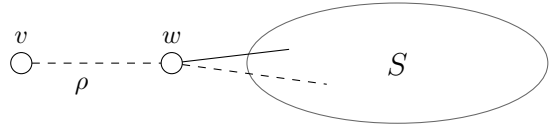
        \draw (2.5, -.5ex) node {\Large $S$};
        \draw[opacity=.6] (2.5, 0) ellipse (1 and .4);

        \hnode{h1}{(3.7, .2)} \hnode{h2}{(4.2, -.3)}
        \line{w}{h1}{} \dline{w}{h2}{}
    \end{coxeter}\end{center}
\end{lemma}

\begin{proof}
    Let us choose arbitrary labels on the dotted edges. Denote by \(\rho\) the label on the dotted edge between \(v\) and \(w\). Direct calculation provides \[\det(\langle v, w, S \rangle) = \det(\langle w, S \rangle) - \rho^2 \det(S).\] Suppose that the diagram \(\langle v, w, S \rangle\) is hyperbolic. If \(\det(S) < 0\), then \[\rho \leqslant \sqrt{\frac{\det(\langle w, S \rangle)}{\det(S)}} = \sqrt{1 + \frac{\det(\langle w, S \rangle) - \det(S)}{\det(S)}} \leqslant 1.\] We get \(\det(S) = 0\) and \(\det\big(\langle w, S \rangle\big) > 0\). Therefore, the diagram \(\langle w, S \rangle\) is superhyperbolic.
\end{proof}

\begin{corollary} \label{corollary: superhyperbolicity of the nine diagrams}
    The diagrams below are superhyperbolic for any \(\rho_1, \rho_2, \rho_3 > 1\).

    
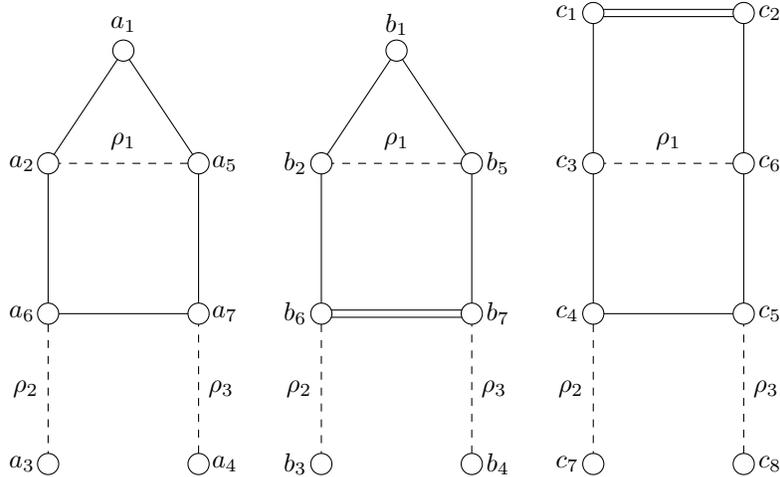
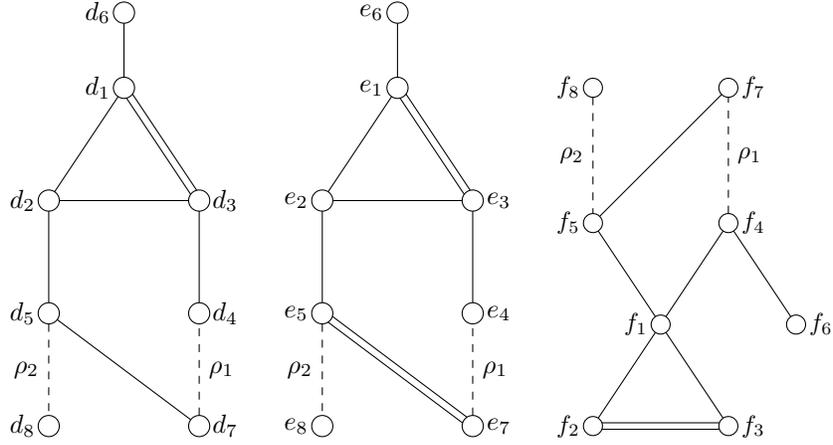
\begin{center}
        \begin{coxeter}[2]
            \def\l{a}
            \node[180]{3}{(0, 0)}{$\l_3$} \node[0]{4}{(2, 0)}{$\l_4$}
            \node[180]{6}{(0, 2)}{$\l_6$} \node[0]{7}{(2, 2)}{$\l_7$}
            \node[180]{2}{(0, 4)}{$\l_2$} \node[0]{5}{(2, 4)}{$\l_5$}
            \node{1}{(1, 5.5)}{$\l_1$}

            \dline{6}{3}{$\rho_2$} \dline{4}{7}{$\rho_3$}
            \line{6}{2}{} \line{7}{5}{}
            \line{1}{2}{} \line{1}{5}{}
            \line{6}{7}{} \dline{5}{2}{$\rho_1$}
        \end{coxeter}
        \begin{coxeter}[2]
            \def\l{b}
            \node[180]{3}{(0, 0)}{$\l_3$} \node[0]{4}{(2, 0)}{$\l_4$}
            \node[180]{6}{(0, 2)}{$\l_6$} \node[0]{7}{(2, 2)}{$\l_7$}
            \node[180]{2}{(0, 4)}{$\l_2$} \node[0]{5}{(2, 4)}{$\l_5$}
            \node{1}{(1, 5.5)}{$\l_1$}

            \dline{6}{3}{$\rho_2$} \dline{4}{7}{$\rho_3$}
            \line{6}{2}{} \line{7}{5}{}
            \line{1}{2}{} \line{1}{5}{}
            \lline{6}{7}{} \dline{5}{2}{$\rho_1$}
        \end{coxeter}
        \begin{coxeter}[2]
            \def\l{c}
            \node[180]{7}{(0, 0)}{$\l_7$} \node[0]{8}{(2, 0)}{$\l_8$}
            \node[180]{4}{(0, 2)}{$\l_4$} \node[0]{5}{(2, 2)}{$\l_5$}
            \node[180]{3}{(0, 4)}{$\l_3$} \node[0]{6}{(2, 4)}{$\l_6$}
            \node[180]{1}{(0, 6)}{$\l_1$} \node[0]{2}{(2, 6)}{$\l_2$}

            \lline{2}{1}{} \line{4}{5}{}
            \line{1}{3}{} \line{2}{6}{}
            \line{3}{4}{} \line{6}{5}{}
            \dline{4}{7}{$\rho_2$} \dline{8}{5}{$\rho_3$}
            \dline{6}{3}{$\rho_1$}
        \end{coxeter}

        \vspace{.5cm}

        \begin{coxeter}[2]
            \def\l{d}
            \node[180]{8}{(0, 0)}{$\l_8$} \node[0]{7}{(2, 0)}{$\l_7$}
            \node[180]{5}{(0, 1.5)}{$\l_5$} \node[0]{4}{(2, 1.5)}{$\l_4$}
            \node[180]{2}{(0, 3)}{$\l_2$} \node[0]{3}{(2, 3)}{$\l_3$}
            \node[180]{1}{(1, 4.5)}{$\l_1$} \node[180]{6}{(1, 5.5)}{$\l_6$}

            \line{1}{2}{} \lline{1}{3} \line{2}{3}{}
            \line{2}{5}{} \line{3}{4}{} \line{1}{6}{}
            \dline{7}{4}{$\rho_1$} \dline{5}{8}{$\rho_2$} \line{5}{7}{}
        \end{coxeter}
        \begin{coxeter}[2]
            \def\l{e}
            \node[180]{8}{(0, 0)}{$\l_8$} \node[0]{7}{(2, 0)}{$\l_7$}
            \node[180]{5}{(0, 1.5)}{$\l_5$} \node[0]{4}{(2, 1.5)}{$\l_4$}
            \node[180]{2}{(0, 3)}{$\l_2$} \node[0]{3}{(2, 3)}{$\l_3$}
            \node[180]{1}{(1, 4.5)}{$\l_1$} \node[180]{6}{(1, 5.5)}{$\l_6$}

            \line{1}{2}{} \lline{1}{3} \line{2}{3}{}
            \line{2}{5}{} \line{3}{4}{} \line{1}{6}{}
            \dline{7}{4}{$\rho_1$} \dline{5}{8}{$\rho_2$} \lline{5}{7}
        \end{coxeter}
        \begin{coxeter}[1.8]
            \def\l{f}
            \node[180]{8}{(0, 5)}{$\l_8$} \node[0]{7}{(2, 5)}{$\l_7$}
            \node[180]{5}{(0, 3)}{$\l_5$} \node[0]{4}{(2, 3)}{$\l_4$}
            \node[180]{2}{(0, 0)}{$\l_2$} \node[0]{3}{(2, 0)}{$\l_3$}
            \node[180]{1}{(1, 1.5)}{$\l_1$} \node[0]{6}{(3, 1.5)}{$\l_6$}

            \line{1}{2}{} \line{1}{3}{} \lline{2}{3}
            \line{1}{5}{} \line{1}{4}{} \line{4}{6}{}
            \dline{4}{7}{$\rho_1$} \dline{8}{5}{$\rho_2$} \line{5}{7}{}
        \end{coxeter}
    \end{center}
\end{corollary}

\begin{proof} For \(\rho_1, \rho_2, \rho_3 > 1\) we have
    \begin{align*}
        \det(\langle a_7, A \rangle) - \det(A) &= \frac{1}{16} \big(3\rho_2^2 + 4\rho_1^2 - 2\rho_1 - 5\big) > 0,\\
        \det(\langle b_7, B \rangle) - \det(B) &= \frac{1}{16} \big(3\rho_2^2 + 8\rho_1^2 - 4(\sqrt{2} - 1)\rho_1 - 6 - \sqrt{2}\big) > 0,\\
        \det(\langle c_5, C \rangle) - \det(C) &= \frac{1}{64} \big(4\rho_2^2 + 8\rho_1^2 - 4(2 - \sqrt{2})\rho_1 - 2\sqrt{2} - 3\big) > 0,&\\
        \det(\langle d_5, D \rangle) - \det(D) &= \frac{1}{32} \big(2\rho_1^2 - (3 + 2\sqrt{2})\rho_1 + 2\sqrt{2} + 2\big) > 0,\\
        \det(\langle e_5, E \rangle) - \det(E) &= \frac{1}{64} \big(4\rho_1^2 - 2(4 + 3\sqrt{2})\rho_1 + 8\sqrt{2} + 9\big) > 0,\\
        \det(\langle f_5, F \rangle) - \det(F) &= \frac{1}{64} \big(8\rho_1^2 - 8\rho_1 + 3\sqrt{2} - 4\big) > 0,
    \end{align*} where \begin{align*}
        A &= \langle a_1, a_2, a_3, a_5, a_6 \rangle, & D &= \langle d_1, d_2, d_3, d_4, d_6, d_7 \rangle,\\
        B &= \langle b_1, b_2, b_3, b_5, b_6 \rangle, & E &= \langle e_1, e_2, e_3, e_4, e_6, e_7 \rangle,\\
        C &= \langle c_1, c_2, c_3, c_4, c_6, c_7 \rangle, & F &= \langle f_1, f_2, f_3, f_4, f_6, f_7 \rangle.
    \end{align*}
\end{proof}

\begin{lemma} \label{lemma: superhyperbolicity of the seven diagrams}
    The diagrams below are superhyperbolic for any \(\rho > 1\).

    
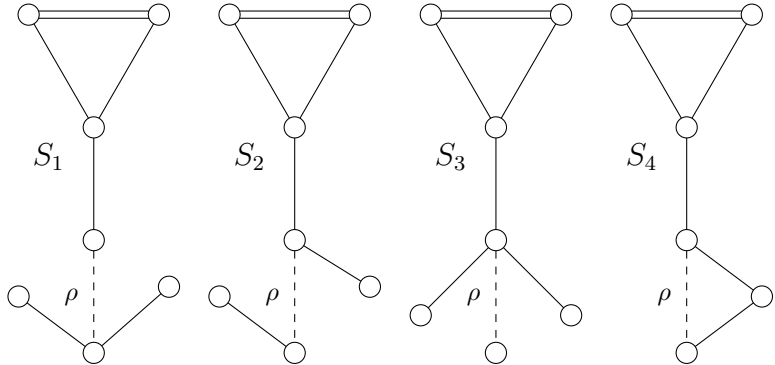
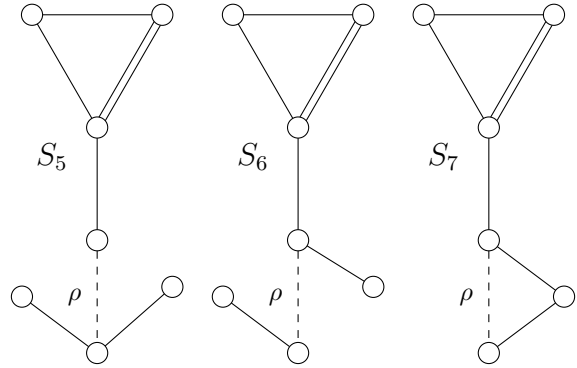
\begin{center}
        \begin{coxeter}[2]
            \draw (-.3, -.75) node {\Large \(S_1\)};
            \node{1}{({cos(270)}, {sin(270)})}{}
            \node{2}{({cos(150)}, {sin(150)})}{}
            \node{3}{({cos(30)}, {sin(30)})}{}
            \node{4}{(0, -2.5)}{} \node{5}{(0, -4)}{}
            \node{6}{(-1, -3.25)}{}
            \node{7}{(1, -3.12)}{}

            \line{1}{2}{} \line{1}{3}{} \lline{2}{3}
            \line{1}{4}{} \dline{4}{5}{$\rho$}
            \line{5}{6}{} \line{5}{7}{}
        \end{coxeter}
        \begin{coxeter}[2]
            \draw (-.3, -.75) node {\Large \(S_2\)};
            \node{1}{({cos(270)}, {sin(270)})}{}
            \node{2}{({cos(150)}, {sin(150)})}{}
            \node{3}{({cos(30)}, {sin(30)})}{}
            \node{4}{(0, -2.5)}{} \node{5}{(0, -4)}{}
            \node{6}{(-1, -3.25)}{}
            \node{7}{(1, -3.12)}{}

            \line{1}{2}{} \line{1}{3}{} \lline{2}{3}
            \line{1}{4}{} \dline{4}{5}{$\rho$}
            \line{5}{6}{} \line{4}{7}{}
        \end{coxeter}
        \begin{coxeter}[2]
            \draw (-.3, -.75) node {\Large \(S_3\)};
            \node{1}{({cos(270)}, {sin(270)})}{}
            \node{2}{({cos(150)}, {sin(150)})}{}
            \node{3}{({cos(30)}, {sin(30)})}{}
            \node{4}{(0, -2.5)}{} \node{5}{(0, -4)}{}
            \node{6}{({cos(180)}, {sin(180) - 3.5})}{}
            \node{7}{({cos(0)}, {sin(0) - 3.5})}{}

            \line{1}{2}{} \line{1}{3}{} \lline{2}{3}
            \line{1}{4}{} \dline{4}{5}{$\rho$}
            \line{4}{6}{} \line{4}{7}{}
        \end{coxeter}
        \begin{coxeter}[2]
            \draw (-.3, -.75) node {\Large \(S_4\)};
            \node{1}{({cos(270)}, {sin(270)})}{}
            \node{2}{({cos(150)}, {sin(150)})}{}
            \node{3}{({cos(30)}, {sin(30)})}{}
            \node{4}{(0, -2.5)}{} \node{5}{(0, -4)}{}
            \node{6}{(1, -3.25)}{}

            \line{1}{2}{} \line{1}{3}{} \lline{2}{3}
            \line{1}{4}{} \dline{4}{5}{$\rho$}
            \line{4}{6}{} \line{5}{6}{}
        \end{coxeter}

        \vspace{.5cm}

        \begin{coxeter}[2]
            \draw (-.3, -.75) node {\Large \(S_5\)};
            \node{1}{({cos(270)}, {sin(270)})}{}
            \node{2}{({cos(150)}, {sin(150)})}{}
            \node{3}{({cos(30)}, {sin(30)})}{}
            \node{4}{(0, -2.5)}{} \node{5}{(0, -4)}{}
            \node{6}{(-1, -3.25)}{}
            \node{7}{(1, -3.12)}{}

            \line{1}{2}{} \lline{1}{3} \line{2}{3}{}
            \line{1}{4}{} \dline{4}{5}{$\rho$}
            \line{5}{6}{} \line{5}{7}{}
        \end{coxeter}
        \begin{coxeter}[2]
            \draw (-.3, -.75) node {\Large \(S_6\)};
            \node{1}{({cos(270)}, {sin(270)})}{}
            \node{2}{({cos(150)}, {sin(150)})}{}
            \node{3}{({cos(30)}, {sin(30)})}{}
            \node{4}{(0, -2.5)}{} \node{5}{(0, -4)}{}
            \node{6}{(-1, -3.25)}{}
            \node{7}{(1, -3.12)}{}

            \line{1}{2}{} \lline{1}{3} \line{2}{3}{}
            \line{1}{4}{} \dline{4}{5}{$\rho$}
            \line{5}{6}{} \line{4}{7}{}
        \end{coxeter}
        \begin{coxeter}[2]
            \draw (-.3, -.75) node {\Large \(S_7\)};
            \node{1}{({cos(270)}, {sin(270)})}{}
            \node{2}{({cos(150)}, {sin(150)})}{}
            \node{3}{({cos(30)}, {sin(30)})}{}
            \node{4}{(0, -2.5)}{} \node{5}{(0, -4)}{}
            \node{6}{(1, -3.25)}{}

            \line{1}{2}{} \lline{1}{3} \line{2}{3}{}
            \line{1}{4}{} \dline{4}{5}{$\rho$}
            \line{4}{6}{} \line{5}{6}{}
        \end{coxeter}
    \end{center}
\end{lemma}

\begin{proof} For \(\rho > 1\) we have
    \begin{align*}
        \det(S_1) &= \frac{1}{16}\Big(4\sqrt{2}\rho^2 - 2\sqrt{2} - 1\Big) > 0,\\
        \det(S_2) &= \frac{1}{64}\Big(16\sqrt{2}\rho^2 - 9\sqrt{2} - 6\Big) > 0,\\
        \det(S_3) &= \frac{1}{8}\Big(2\sqrt{2}\rho^2 - \sqrt{2} - 1\Big) > 0,\\
        \det(S_4) &= \frac{1}{32}\Big(8\sqrt{2}\rho^2 + 4\sqrt{2}\rho - 4\sqrt{2} - 3\Big) > 0,\\        \det(S_5) &= \frac{1}{32}\Big(8\sqrt{2}\rho^2 - 4\sqrt{2} - 3\Big) > 0,\\
        \det(S_6) &= \frac{1}{64}\Big(16\sqrt{2}\rho^2 - 9\sqrt{2} - 9\Big) > 0,\\
        \det(S_7) &= \frac{1}{64}\Big(16\sqrt{2}\rho^2 + 8\sqrt{2}\rho - 8\sqrt{2} - 9\Big) > 0.
    \end{align*}
\end{proof}

\begin{lemma}\label{lemma: superhyperbolicity of the diagram}
    The diagrams below are superhyperbolic for any \(\rho > 1\).
    
    \begin{center}
    	\def\s{1.8}
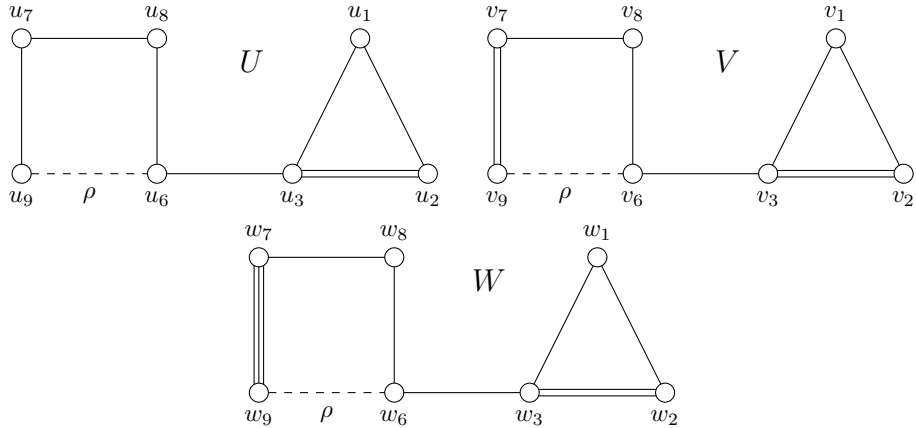
	    \begin{coxeter}[\s]
    	    \node[-90]{9}{(0, 0)}{$u_9$} \node[-90]{6}{(2, 0)}{$u_6$}
        	\node{7}{(0, 2)}{$u_7$} \node{8}{(2, 2)}{$u_8$}
	        \node[-90]{3}{(4, 0)}{$u_3$} \node[-90]{2}{(6, 0)}{$u_2$}
    	    \node{1}{(5, 2)}{$u_1$}
    	    
    	    \draw (1.7, .75) node {\Large $U$};

        	\dline{9}{6}{$\rho$} \line{6}{3}{}
	        \line{9}{7}{} \line{7}{8}{} \line{8}{6}{}
    	    \line{3}{1}{} \line{1}{2}{} \lline{2}{3}
	    \end{coxeter}
	    \begin{coxeter}[\s]
    	    \node[-90]{9}{(0, 0)}{$v_9$} \node[-90]{6}{(2, 0)}{$v_6$}
        	\node{7}{(0, 2)}{$v_7$} \node{8}{(2, 2)}{$v_8$}
	        \node[-90]{3}{(4, 0)}{$v_3$} \node[-90]{2}{(6, 0)}{$v_2$}
    	    \node{1}{(5, 2)}{$v_1$}

    	    \draw (1.7, .75) node {\Large $V$};

        	\dline{9}{6}{$\rho$} \line{6}{3}{}
	        \lline{9}{7} \line{7}{8}{} \line{8}{6}{}
    	    \line{3}{1}{} \line{1}{2}{} \lline{2}{3}
	    \end{coxeter}
	    \begin{coxeter}[\s]
    	    \node[-90]{9}{(0, 0)}{$w_9$} \node[-90]{6}{(2, 0)}{$w_6$}
        	\node{7}{(0, 2)}{$w_7$} \node{8}{(2, 2)}{$w_8$}
	        \node[-90]{3}{(4, 0)}{$w_3$} \node[-90]{2}{(6, 0)}{$w_2$}
    	    \node{1}{(5, 2)}{$w_1$}

    	    \draw (1.7, .75) node {\Large $W$};

        	\dline{9}{6}{$\rho$} \line{6}{3}{}
	        \llline{9}{7} \line{7}{8}{} \line{8}{6}{}
    	    \line{3}{1}{} \line{1}{2}{} \lline{2}{3}
	    \end{coxeter}
    \end{center}
\end{lemma}

\begin{proof} For \(\rho > 1\) we have
    \begin{align*}
    	\det(U) &= \frac{1}{64} \left(12\sqrt{2}\rho^2 + 4\sqrt{2}\rho - 5\sqrt{2} - 6\right) > 0, \\
    	\det(V) &= \frac{1}{64} \left(12\sqrt{2}\rho^2 + 8\rho - 2\sqrt{2} - 3\right) > 0, \\
    	\det(W) &= \frac{1}{128} \left(24\sqrt{2}\rho^2 + 4\sqrt{2}(1 + \sqrt{5})\rho + 3\sqrt{10} + 3\sqrt{5} - 7\sqrt{2} - 9\right) > 0.
    \end{align*}
\end{proof}

Let us remind that we suppose that \([u_8, L_1] = [u_9, L_1] = 0\), \([u_4, L_1] \geqslant [u_7, L_1]\), and \([u_5, L_1] = [u_6, L_1] = 1\). Thus, the subdiagram \(\langle L_1, u_4, u_5, u_6 \rangle\) is one of the following.

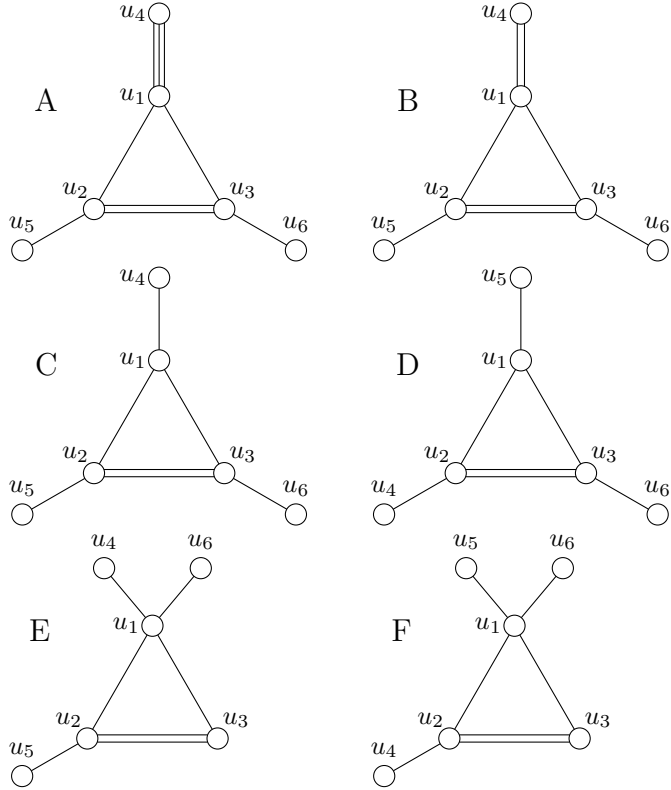
\begin{center}\begin{tabular}{l l}
    \begin{coxeter}[2]
        \def\l{1.1}
        \draw (-.75, .4) node {\Large A};
        \node[180]{1}{({cos(90)}, {sin(90)})}{$u_1$}
        \node[135]{2}{({cos(210)}, {sin(210)})}{$u_2$}
        \node[45]{3}{({cos(330)}, {sin(330)})}{$u_3$}
        \node[180]{4}{({cos(90)}, {sin(90) + \l})}{$u_4$}
        \node{5}{({(1 + \l)*cos(210)}, {(1 + \l)*sin(210)})}{$u_5$}
        \node{6}{({(1 + \l)*cos(330)}, {(1 + \l)*sin(330)})}{$u_6$}
        \line{1}{2}{} \lline{2}{3} \line{3}{1}{}
        \llline{1}{4} \line{5}{2}{} \line{3}{6}{}
    \end{coxeter} &
    \begin{coxeter}[2]
        \def\l{1.1}
        \draw (-.75, .4) node {\Large B};
        \node[180]{1}{({cos(90)}, {sin(90)})}{$u_1$}
        \node[135]{2}{({cos(210)}, {sin(210)})}{$u_2$}
        \node[45]{3}{({cos(330)}, {sin(330)})}{$u_3$}
        \node[180]{4}{({cos(90)}, {sin(90) + \l})}{$u_4$}
        \node{5}{({(1 + \l)*cos(210)}, {(1 + \l)*sin(210)})}{$u_5$}
        \node{6}{({(1 + \l)*cos(330)}, {(1 + \l)*sin(330)})}{$u_6$}
        \line{1}{2}{} \lline{2}{3} \line{3}{1}{}
        \lline{1}{4} \line{5}{2}{} \line{3}{6}{}
    \end{coxeter} \\

    \begin{coxeter}[2]
        \def\l{1.1}
        \draw (-.75, .4) node {\Large C};
        \node[180]{1}{({cos(90)}, {sin(90)})}{$u_1$}
        \node[135]{2}{({cos(210)}, {sin(210)})}{$u_2$}
        \node[45]{3}{({cos(330)}, {sin(330)})}{$u_3$}
        \node[180]{4}{({cos(90)}, {sin(90) + \l})}{$u_4$}
        \node{5}{({(1 + \l)*cos(210)}, {(1 + \l)*sin(210)})}{$u_5$}
        \node{6}{({(1 + \l)*cos(330)}, {(1 + \l)*sin(330)})}{$u_6$}
        \line{1}{2}{} \lline{2}{3} \line{3}{1}{}
        \line{1}{4}{} \line{5}{2}{} \line{3}{6}{}
    \end{coxeter} &
    \begin{coxeter}[2]
        \def\l{1.1}
        \draw (-.75, .4) node {\Large D};
        \node[180]{1}{({cos(90)}, {sin(90)})}{$u_1$}
        \node[135]{2}{({cos(210)}, {sin(210)})}{$u_2$}
        \node[45]{3}{({cos(330)}, {sin(330)})}{$u_3$}
        \node[180]{4}{({cos(90)}, {sin(90) + \l})}{$u_5$}
        \node{5}{({(1 + \l)*cos(210)}, {(1 + \l)*sin(210)})}{$u_4$}
        \node{6}{({(1 + \l)*cos(330)}, {(1 + \l)*sin(330)})}{$u_6$}
        \line{1}{2}{} \lline{2}{3} \line{3}{1}{}
        \line{1}{4}{} \line{5}{2}{} \line{3}{6}{}
    \end{coxeter} \\
    
    \begin{coxeter}[2]
        \def\l{1}
        \draw (-.75, .4) node {\Large E};
        \node[180]{1}{({cos(90)}, {sin(90)})}{$u_1$}
        \node[135]{2}{({cos(210)}, {sin(210)})}{$u_2$}
        \node[45]{3}{({cos(330)}, {sin(330)})}{$u_3$}
        \node{5}{({(1 + \l)*cos(210)}, {(1 + \l)*sin(210)})}{$u_5$}
        \node{4}{({cos(90) + \l*cos(130)}, {sin(90) + \l*sin(130)})}{$u_4$}
        \node{6}{({cos(90) + \l*cos(50)}, {sin(90) + \l*sin(50)})}{$u_6$}
        \line{1}{2}{} \lline{2}{3} \line{3}{1}{}
        \line{1}{4}{} \line{5}{2}{} \line{1}{6}{}
    \end{coxeter} &
    \begin{coxeter}[2]
        \def\l{1}
        \draw (-.75, .4) node {\Large F};
        \node[180]{1}{({cos(90)}, {sin(90)})}{$u_1$}
        \node[135]{2}{({cos(210)}, {sin(210)})}{$u_2$}
        \node[45]{3}{({cos(330)}, {sin(330)})}{$u_3$}
        \node{5}{({(1 + \l)*cos(210)}, {(1 + \l)*sin(210)})}{$u_4$}
        \node{4}{({cos(90) + \l*cos(130)}, {sin(90) + \l*sin(130)})}{$u_5$}
        \node{6}{({cos(90) + \l*cos(50)}, {sin(90) + \l*sin(50)})}{$u_6$}
        \line{1}{2}{} \lline{2}{3} \line{3}{1}{}
        \line{1}{4}{} \line{5}{2}{} \line{1}{6}{}
    \end{coxeter}
\end{tabular}\end{center}

\subsubsection{Case A}

\begin{align*}
    [u_5, u_4] = [u_5, u_7] = [u_4, u_8] = 0, \\
    [u_4, u_6] = [u_4, u_9] = [u_6, u_7] = 0, \\
    [u_5, u_6] = [u_5, u_9] = [u_6, u_8] = 0.
\end{align*} 
Otherwise, there is either a parabolic or hyperbolic subdiagram that must be elliptic. This implies that \([u_7, u_8] \ne 0\), \([u_8, u_9] \ne 0\), and \([u_9, u_7] \ne 0\). Then the subdiagram \(\langle u_7, u_8, u_9 \rangle\) is not elliptic.

\subsubsection{Case B}

By the same argument we get 
\begin{align*}
    [u_5, u_4] = [u_5, u_7] &= [u_4, u_8] = 0, \\
    [u_4, u_6] = [u_4, u_9] &= [u_6, u_7] = 0, \\
    [u_5, u_6] &= 0.
\end{align*} 
This yields that, without loss of generality, 
\begin{align*}
    1 \leqslant [u_7, u_8], \quad 1 \leqslant [u_7, u_9] \leqslant 3, \quad & [u_8, u_9] = 0, \\
    [u_6, u_8] = 1, \quad [u_7, u_8] = 1, \quad & [u_5, u_9] \in \{0, 1\}, \\
    [u_7, u_1] = [u_7, u_2] = [u_7, u_3]& = 0.
\end{align*} 
Therefore, the diagram is equal to the shown below.

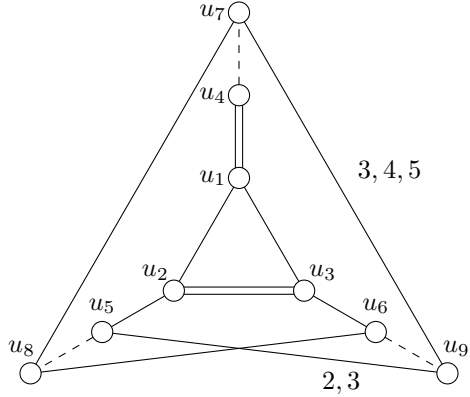
\begin{center}\begin{coxeter}[2]
    \def\l{1.1}
    \node[180]{1}{({cos(90)}, {sin(90)})}{$u_1$}
    \node[135]{2}{({cos(210)}, {sin(210)})}{$u_2$}
    \node[45]{3}{({cos(330)}, {sin(330)})}{$u_3$}
    \node[180]{4}{({cos(90)}, {sin(90) + \l})}{$u_4$}
    \node{5}{({(1 + \l)*cos(210)}, {(1 + \l)*sin(210)})}{$u_5$}
    \node{6}{({(1 + \l)*cos(330)}, {(1 + \l)*sin(330)})}{$u_6$}
    \node[180]{7}{({cos(90)}, {sin(90) + 2*\l})}{$u_7$}
    \node[110]{8}{({(1 + 2*\l)*cos(210)}, {(1 + 2*\l)*sin(210)})}{$u_8$}
    \node[70]{9}{({(1 + 2*\l)*cos(330)}, {(1 + 2*\l)*sin(330)})}{$u_9$}

    \line{1}{2}{} \lline{2}{3} \line{3}{1}{}
    \lline{1}{4} \line{5}{2}{} \line{3}{6}{}
    \dline{4}{7}{} \dline{5}{8}{} \dline{6}{9}{}
    \line{7}{8}{} \line[.5]{9}{7}{} \line{6}{8}{} \line[.7]{5}{9}{$2, 3$}
    \draw (1, .5) node{$3, 4, 5$};
\end{coxeter}\end{center}
From Lemma~\ref{lemma: superhyperbolicity of the diagram} it follows that the subdiagram \(\langle u_1, u_2, u_3, u_6, u_7, u_8, u_9 \rangle\) is superhyperbolic.

\subsubsection{Case C}

\[
	[u_5, u_4] = [u_5, u_7] = [u_6, u_4] = [u_6, u_7] = 0.
\]

Let \([u_7, u_1] = 0\). Suppose that \([u_8, u_7] = 0\). Then
\[
	1 \leqslant [u_4, u_8] \leqslant 2, \quad [u_6, u_5] = [u_6, u_8] = 0.
\]
Corollary \ref{corollary: superhyperbolicity of the nine diagrams} (\(D\) and \(E\)) 
implies that the diagram \(\langle u_1, u_2, u_3, u_4, u_5, u_6, u_7, u_8 \rangle\) is superhyperbolic. Therefore, \([u_8, u_7] \geqslant 1\). For similar reasons, \([u_9, u_7] \geqslant 1\). Without loss of generality, \([u_8, u_6] = 1\), \([u_8, u_7] = 1\), and \(1 \leqslant [u_9, u_7] \leqslant 3\). The subdiagram \(\langle u_1, u_2, u_3, u_6, u_7, u_8, u_9 \rangle\) is superhyperbolic due to Lemma \ref{lemma: superhyperbolicity of the diagram}.

\medskip

Let \([u_7, u_1] = 1\), then the only possible diagram is shown  below.
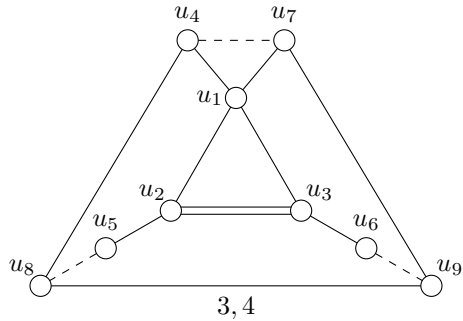
\begin{center}\begin{coxeter}[2]
    \def\l{1}
    \node[180]{1}{({cos(90)}, {sin(90)})}{$u_1$}
    \node[135]{2}{({cos(210)}, {sin(210)})}{$u_2$}
    \node[45]{3}{({cos(330)}, {sin(330)})}{$u_3$}
    \node{4}{({cos(90) + \l*cos(130)}, {sin(90) + \l*sin(130)})}{$u_4$}
    \node[90]{5}{({(1 + \l)*cos(210)}, {(1 + \l)*sin(210)})}{$u_5$}
    \node[90]{6}{({(1 + \l)*cos(-30)}, {(1 + \l)*sin(-30)})}{$u_6$}
    \node{7}{({cos(90) + \l*cos(50)}, {sin(90) + \l*sin(50)})}{$u_7$}
    \node[135]{8}{({(1 + 2*\l)*cos(210)}, {(1 + 2*\l)*sin(210)})}{$u_8$}
    \node[45]{9}{({(1 + 2*\l)*cos(-30)}, {(1 + 2*\l)*sin(-30)})}{$u_9$}

    \line{1}{2}{} \lline{2}{3} \line{3}{1}{}
    \line{1}{4}{} \line{1}{7}{} \line{2}{5}{} \line{3}{6}{}
    \dline{4}{7}{} \dline{5}{8}{} \dline{6}{9}{}
    \line{4}{8}{} \line{7}{9}{} \line{8}{9}{$3, 4$}
\end{coxeter}\end{center}
Corollary \ref{corollary: superhyperbolicity of the nine diagrams} (\(A\) and \(B\)) implies that the subdiagram \(\langle u_1, u_4, u_5, u_6, u_7, u_8, u_9 \rangle\) is superhyperbolic.

\subsubsection{Case D}

The case \([u_7, L_1] = 0\) is considered in the previous paragraph, so \([u_7, L_1] \ne 0\). Moreover, \([u_7, u_3] = 0\). Suppose that \([u_7, u_2] \geqslant 1\). Then the diagram \(\langle L_2, L_3 \rangle\) is not connected. Therefore, \([u_7, u_2] = 0\) and \([u_7, u_1] = 1\). The equality
\[[u_4, u_5] = [u_4, u_8] = [u_7, u_5] = 0\]
implies that \([u_7, u_8] \ne 0\). It is easy to check that 
\[[u_4, u_6] = [u_7, u_6] = [u_5, u_6] = [u_8, u_6] = 0.\]
Suppose that \([u_5, u_9] \geqslant 1\). Then \([L_2, u_9] = 0\) and the subdiagram \(\langle L_2, L_4 \rangle\) is not connected. Therefore, \([u_5, u_9] = 0\), \([u_8, u_9] \geqslant 1\), \([u_7, u_9] = 0\), and \([u_4, u_9] \geqslant 1\). The only possible diagram is shown below.

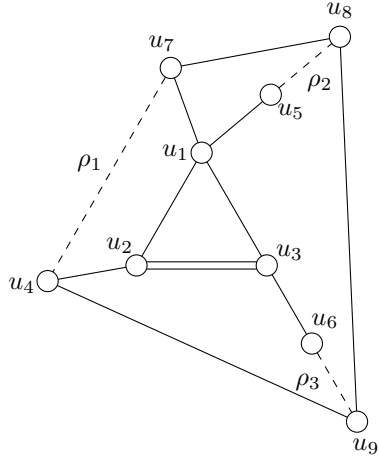
\begin{center}\begin{coxeter}[2]
    \def\l{1.2}
    \def\a{300} \def\b{300} \def\g{40} \def\d{40}
    \node[180]{1}{({cos(90)}, {sin(90)})}{$u_1$}
    \node[130]{2}{({cos(210)}, {sin(210)})}{$u_2$}
    \node[30]{3}{({cos(330)}, {sin(330)})}{$u_3$}
    \node[190]{4}{({cos(210) + \l*cos(190)}, {sin(210) + \l*sin(190)})}{$u_4$}
    \node[-40]{5}{({cos(90) + \l*cos(\g)}, {sin(90) + \l*sin(\g)})}{$u_5$}
    \node[60]{6}{({cos(330) + \l*cos(\a)}, {sin(330) + \l*sin(\a)})}{$u_6$}
    \node[110]{7}{({cos(90) + \l*cos(110)}, {sin(90) + \l*sin(110)})}{$u_7$}
    \node{8}{({cos(90) + \l*cos(\g) + \l*cos(\d)}, {sin(90) + \l*sin(\g) + \l*sin(\d)})}{$u_8$}
    \node[-70]{9}{({cos(330) + \l*cos(\a) + \l*cos(\b)}, {sin(330) + \l*sin(\a) + \l*sin(\b)})}{$u_9$}

    \line{1}{2}{} \line{1}{3}{} \lline{2}{3}
    \dline{7}{4}{$\rho_1$} \dline{5}{8}{$\rho_2$} \dline[.35]{6}{9}{$\rho_3$}
    \line{2}{4}{} \line{1}{7}{} \line{1}{5}{} \line{3}{6}{}
    \line{4}{9}{} \line{7}{8}{} \line{8}{9}{}
\end{coxeter}\end{center}
But this diagram is superhyperbolic since \[\det(\langle u_1, u_2, u_3, u_5, u_7, u_8, u_9 \rangle) = \frac{1}{32}\Big(4\big(2\sqrt{2} + 1\big)\rho_2^2 - 4\rho_2 - \big(4\sqrt{2} + 5\big)\Big) > 0\] for all \(\rho_2 > 1\).

\subsubsection{Case E}

Let \([u_7, L_1] = 0\). Lemma~\ref{lemma: superhyperbolic diagram from lanner 3-diagram, 2-diagram, and a vertex} implies that the diagrams below are superhyperbolic.
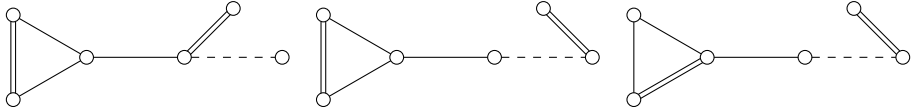
\begin{center}
	\def\s{1.3}
	\begin{coxeter}[\s]
    	\node{1}{({cos(120)}, {sin(120)})}{}
	    \node{2}{({cos(240)}, {sin(240)})}{}
    	\node{3}{(1, 0)}{}\node{4}{(3, 0)}{}
	    \node{5}{(4, 1)}{}\node{6}{(5, 0)}{}
    	\lline{1}{2}\line{2}{3}{}\line{3}{1}{}
	    \line{4}{3}{}\lline{5}{4}{}
    	\dline{4}{6}{}
	\end{coxeter}
	\begin{coxeter}[\s]
    	\node{1}{({cos(120)}, {sin(120)})}{}
	    \node{2}{({cos(240)}, {sin(240)})}{}
    	\node{3}{(1, 0)}{}\node{4}{(3, 0)}{}
	    \node{5}{(4, 1)}{}\node{6}{(5, 0)}{}
    	\lline{1}{2}\line{2}{3}{}\line{3}{1}{}
	    \line{4}{3}{}\lline{6}{5}{}
    	\dline{4}{6}{}
	\end{coxeter}
	\begin{coxeter}[\s]
    	\node{1}{({cos(120)}, {sin(120)})}{}
	    \node{2}{({cos(240)}, {sin(240)})}{}
    	\node{3}{(1, 0)}{}\node{4}{(3, 0)}{}
	    \node{5}{(4, 1)}{}\node{6}{(5, 0)}{}
    	\line{1}{2}{}\lline{2}{3}\line{3}{1}{}
	    \line{4}{3}{}\lline{6}{5}
    	\dline{4}{6}{}
	\end{coxeter}
\end{center}
The diagram
\begin{coxeter}
	\node{1}{({cos(120)}, {sin(120)})}{}
	\node{2}{({cos(240)}, {sin(240)})}{}
    \node{3}{(1, 0)}{}\node{4}{(3, 0)}{}
	\node{5}{(4, 1)}{}\node{6}{(5, 0)}{}
    \line{1}{2}{}\lline{2}{3}\line{3}{1}{}
	\line{4}{3}{}\lline{5}{4}{}
    \dline{4}{6}{}
\end{coxeter}
contains a parabolic subdiagram. Using Remark~\ref{remark: monotonicity}, we get that if \(k \geqslant 4\) or \(l \geqslant 4\), then the diagram below is superhyperbolic for any \(\rho > 1\).
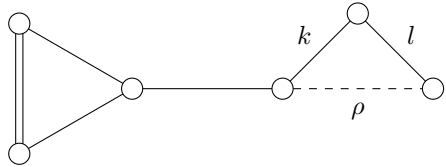
\begin{center}\begin{coxeter}[2]
    \node{1}{({cos(120)}, {sin(120)})}{}
    \node{2}{({cos(240)}, {sin(240)})}{}
    \node{3}{(1, 0)}{}\node{4}{(3, 0)}{}
    \node{5}{(4, 1)}{}\node{6}{(5, 0)}{}
    \lline{1}{2}\line{2}{3}{}\line{3}{1}{}
    \line{4}{3}{}\line{5}{4}{$k$}\line{6}{5}{$l$}
    \dline{4}{6}{$\rho$}
\end{coxeter}\end{center}
By the same argument, if \(k \geqslant 4\) or \(l \geqslant 4\), then the diagram below either contains an unwanted parabolic or Lann\'er subdiagram or is superhyperbolic for any \(\rho > 1\).
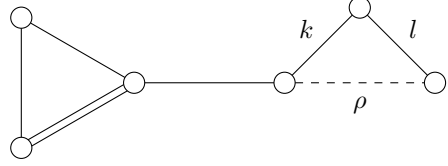
\begin{center}\begin{coxeter}[2]
    \node{1}{({cos(120)}, {sin(120)})}{}
    \node{2}{({cos(240)}, {sin(240)})}{}
    \node{3}{(1, 0)}{}\node{4}{(3, 0)}{}
    \node{5}{(4, 1)}{}\node{6}{(5, 0)}{}
    \line{1}{2}{}\lline{2}{3}\line{3}{1}{}
    \line{4}{3}{}\line{5}{4}{$k$}\line{6}{5}{$l$}
    \dline{4}{6}{$\rho$}
\end{coxeter}\end{center}
Therefore, the multiplicity of every edge between the subdiagrams \(L_2\), \(L_3\), and \(L_4\) does no exceed \(1\).

Applying Lemma~\ref{lemma: superhyperbolicity of the seven diagrams} (\(S_1\)--\(S_4\)) to the subdiagram \(\langle u_1, u_2, u_3, u_4, u_7, u_8, u_9 \rangle\), we obtain that
\[[u_7, u_8] = [u_4, u_8] = 0 \quad \text{or} \quad [u_7, u_9] = [u_4, u_9] = 0.\]
By the same argument,
\[[u_9, u_8] = [u_6, u_8] = 0 \quad \text{or} \quad [u_9, u_7] = [u_6, u_7] = 0.\]
Note that the diagram below contains a parabolic subdiagram.
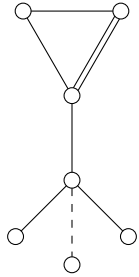
\begin{center}
	\begin{coxeter}[1.5]
        \node{1}{({cos(270)}, {sin(270)})}{}
        \node{2}{({cos(150)}, {sin(150)})}{}
        \node{3}{({cos(30)}, {sin(30)})}{}
        \node{4}{(0, -2.5)}{} \node{5}{(0, -4)}{}
        \node{6}{({cos(180)}, {sin(180) - 3.5})}{}
        \node{7}{({cos(0)}, {sin(0) - 3.5})}{}

        \line{1}{2}{} \lline{1}{3} \line{2}{3}{}
        \line{1}{4}{} \dline{4}{5}{}
        \line{4}{6}{} \line{4}{7}{}
    \end{coxeter}
\end{center}
Thus, applying Lemma~\ref{lemma: superhyperbolicity of the seven diagrams} (\(S_5\)--\(S_7\)) to the subdiagram \(\langle u_1, u_2, u_3, u_5, u_7, u_8, u_9 \rangle\), we obtain that
\[[u_8, u_7] = [u_5, u_7] = 0 \quad \text{or} \quad [u_8, u_9] = [u_5, u_9] = 0.\]
It is easy to check that, without loss of generality, the only diagram with such properties is shown below.
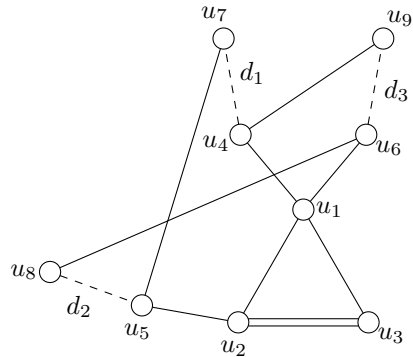
\begin{center}\begin{coxeter}[2]
    \def\l{1.3}
    \node[0]{1}{({cos(90)}, {sin(90)})}{$u_1$}
    \node[260]{2}{({cos(210)}, {sin(210)})}{$u_2$}
    \node[-30]{3}{({cos(330)}, {sin(330)})}{$u_3$}
    \node[200]{4}{({cos(90) + \l*cos(130)}, {sin(90) + \l*sin(130)})}{$u_4$}
    \node[-100]{5}{({cos(210) + \l*cos(170)}, {sin(210) + \l*sin(170)})}{$u_5$}
    \node[-30]{6}{({cos(90) + \l*cos(50)}, {sin(90) + \l*sin(50)})}{$u_6$}
    \node[110]{7}{({cos(90) + \l*cos(130) + \l*cos(100)}, {sin(90) + \l*sin(130) + \l*sin(100)})}{$u_7$}
    \node[180]{8}{({cos(210) + \l*cos(170) + \l*cos(160)}, {sin(210) + \l*sin(170) + \l*sin(160)})}{$u_8$}
    \node[50]{9}{({cos(90) + \l*cos(50) + \l*cos(80)}, {sin(90) + \l*sin(50) + \l*sin(80)})}{$u_9$}

    \line{1}{2}{} \line{1}{3}{} \lline{2}{3}
    \dline[.6]{4}{7}{$d_1$} \dline[.4]{8}{5}{$d_2$} \dline{6}{9}{$d_3$}
    \line{2}{5}{} \line{1}{4}{} \line{1}{6}{}
    \line{6}{8}{} \line{5}{7}{} \line[.4]{9}{4}{}
\end{coxeter}\end{center}
But Corollary~\ref{corollary: superhyperbolicity of the nine diagrams} (\(F\)) implies that the subdiagram \(\langle u_1, u_2, u_3, u_4, u_6, u_7, u_8, u_9 \rangle\) is superhyperbolic.

\medskip

Let \([u_7, L_1] \ne 0\). Then \([u_7, u_2] = 0\). We also may suppose that \([u_7, u_1] \ne 0\) since \([u_7, u_3] \ne 0\) is already considered in Case~D.
\[[u_5, u_4] = [u_5, u_7] = [u_6, u_4] = [u_6, u_7] = 0.\]
Without loss of generality, the only such diagram is shown below.

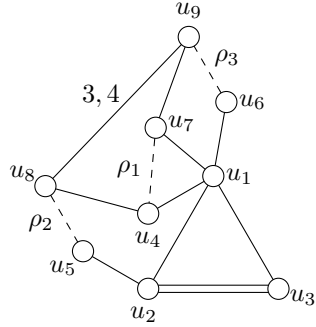
\begin{center}\begin{coxeter}[2]
    \def\l{1}
    \node[0]{1}{({cos(90)}, {sin(90)})}{$u_1$}
    \node[260]{2}{({cos(210)}, {sin(210)})}{$u_2$}
    \node[-20]{3}{({cos(330)}, {sin(330)})}{$u_3$}
    \node[270]{4}{({cos(90) + \l*cos(210)}, {sin(90) + \l*sin(210)})}{$u_4$}
    \node[220]{5}{({cos(210) + \l*cos(150)}, {sin(210) + \l*sin(150)})}{$u_5$}
    \node[0]{6}{({cos(90) + \l*cos(80)}, {sin(90) + \l*sin(80)})}{$u_6$}
    \node[10]{7}{({cos(90) + \l*cos(140)}, {sin(90) + \l*sin(140)})}{$u_7$}
    \node[150]{8}{({cos(210) + \l*cos(150) + \l*cos(120)}, {sin(210) + \l*sin(150) + \l*sin(120)})}{$u_8$}
    \node{9}{({cos(90) + \l*cos(80) + \l*cos(120)}, {sin(90) + \l*sin(80) + \l*sin(120)})}{$u_9$}

    \line{1}{2}{} \lline{2}{3} \line{3}{1}{}
    \line{1}{4}{} \line{5}{2}{} \line{1}{6}{} \line{1}{7}{}
    \dline{7}{4}{$\rho_1$} \dline[.4]{8}{5}{$\rho_2$} \dline{6}{9}{$\rho_3$}
    \line{4}{8}{} \line{7}{9}{} \line{9}{8}{$3, 4$}
\end{coxeter}\end{center}
It is easy to calculate that for \(\rho > 1\)
\[
	\det(\langle u_1, u_2, u_3, u_6, u_7, u_8, u_9 \rangle) = \frac{1}{32}\Big(4(1 + 2\sqrt{2})\rho_3^2 - 4\rho_3 - 4\sqrt{2} - 5\Big) > 0
\]
and
\[
	\det(\langle u_1, u_2, u_3, u_6, u_7, u_8, u_9 \rangle) = \frac{1}{32}\Big(4(1 + 2\sqrt{2})\rho_3^2 - 4\rho_3 - 2\sqrt{2} - 3\Big) > 0.
\]

\subsubsection{Case F}

Let \([u_7, L_1] \ne 0\). The opposite is considered in Case~E. The only such diagrams are shown below.

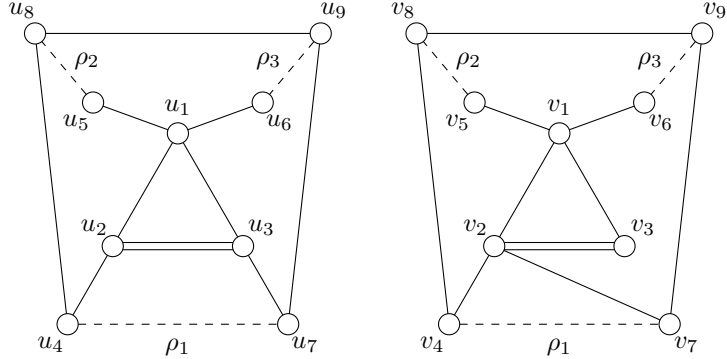
\begin{center}
    \begin{coxeter}[2]
        \def\d{1.2} \def\l{u}
        \node{1}{({cos(90)}, {sin(90)})}{$\l_1$}
        \node[135]{2}{({cos(210)}, {sin(210)})}{$\l_2$}
        \node[45]{3}{({cos(330)}, {sin(330)})}{$\l_3$}
        \node[230]{4}{({cos(210) + \d*cos(240)}, {sin(210) + \d*sin(240)})}{$\l_4$}
        \node[230]{5}{({cos(90) + \d*cos(160)}, {sin(90) + \d*sin(160)})}{$\l_5$}
        \node[310]{6}{({cos(90) + \d*cos(20)}, {sin(90) + \d*sin(20)})}{$\l_6$}
        \node[310]{7}{({cos(330) + \d*cos(300)}, {sin(330) + \d*sin(300)})}{$\l_7$}
        \node[120]{8}{({cos(90) + \d*cos(160) + \d*cos(130)}, {sin(90) + \d*sin(160) + \d*sin(130)})}{$\l_8$}
        \node[60]{9}{({cos(90) + \d*cos(20) + \d*cos(50)}, {sin(90) + \d*sin(20) + \d*sin(50)})}{$\l_9$}
        
        \line{1}{2}{} \lline{2}{3} \line{3}{1}{}
        \line{1}{5}{} \line{4}{2}{} \line{1}{6}{} \line{3}{7}{}
        \dline{4}{7}{$\rho_1$} \dline[.4]{5}{8}{$\rho_2$} \dline[.6]{9}{6}{$\rho_3$}
        \line{8}{9}{} \line{4}{8}{} \line{7}{9}{}
    \end{coxeter}
    \begin{coxeter}[2]
        \def\d{1.2} \def\l{v}
        \node{1}{({cos(90)}, {sin(90)})}{$\l_1$}
        \node[135]{2}{({cos(210)}, {sin(210)})}{$\l_2$}
        \node[45]{3}{({cos(330)}, {sin(330)})}{$\l_3$}
        \node[230]{4}{({cos(210) + \d*cos(240)}, {sin(210) + \d*sin(240)})}{$\l_4$}
        \node[230]{5}{({cos(90) + \d*cos(160)}, {sin(90) + \d*sin(160)})}{$\l_5$}
        \node[310]{6}{({cos(90) + \d*cos(20)}, {sin(90) + \d*sin(20)})}{$\l_6$}
        \node[310]{7}{({cos(330) + \d*cos(300)}, {sin(330) + \d*sin(300)})}{$\l_7$}
        \node[120]{8}{({cos(90) + \d*cos(160) + \d*cos(130)}, {sin(90) + \d*sin(160) + \d*sin(130)})}{$\l_8$}
        \node[60]{9}{({cos(90) + \d*cos(20) + \d*cos(50)}, {sin(90) + \d*sin(20) + \d*sin(50)})}{$\l_9$}
        
        \line{1}{2}{} \lline{2}{3} \line{3}{1}{}
        \line{1}{5}{} \line{4}{2}{} \line{1}{6}{} \line{2}{7}{}
        \dline{4}{7}{$\rho_1$} \dline[.4]{5}{8}{$\rho_2$} \dline[.6]{9}{6}{$\rho_3$}
        \line{8}{9}{} \line{4}{8}{} \line{7}{9}{}
    \end{coxeter}
\end{center}
Corollary \ref{corollary: superhyperbolicity of the nine diagrams} (\(C\) and \(A\)) implies that the subdiagrams \(\langle u_2, u_3, u_4, u_5, u_6, u_7, u_8, u_9 \rangle\) and \(\langle v_2, v_4, v_5, v_6, v_7, v_8, v_9 \rangle\) are superhyperbolic.

\section{Proof of Theorem C}

We say that a polytope is \emph{\(3\)-free} if every set of facets with an empty intersection contains a pair of disjoint facets. Proposition~\ref{proposition: combinatorial structure of coxeter polytopes} implies that the Coxeter diagram of a compact \(3\)-free Coxeter polytope contains no Lann\'er subdiagrams of order \(\geqslant 3\). Our aim is to prove Theorem~\ref{theorem: 3-free}.

The proof is similar to the proof of \cite[Theorem~9.4]{Bur22}, which is based on the proof of \cite[Theorem~6.1]{V85}. Thus, we need the Nikulin inequality.

\begin{theorem}[{\cite[Theorem 3.2.1]{N81}}]
    Let \(\theta_0, \dots, \theta_{k - 1}\) be non-negative reals, \(k \leqslant \left\lfloor \frac{d}{2} \right\rfloor\), and \(P\) a \(d\)-dimensional convex polytope. The following inequality holds \[\frac{1}{\alpha_k^P} \sum_{\substack{Q < P \\ \dim Q = k}} \sum_{i = 0}^{k - 1} \theta_i \alpha_i^Q < \sum_{i = 0}^{k - 1} \theta_i A_d^{(i, k)},\] where \(\alpha_k^R\) is a number of \(k\)-dimensional faces of a polytope \(R\), the notation \(Q < P\) means that \(Q\) is a face of \(P\), and 
    \[
    	A_d^{(i, k)} = \binom{d - i}{k - i} \cdot \frac{\binom{\lceil d / 2 \rceil}{i} + \binom{\lfloor d / 2 \rfloor}{i}}{\binom{\lceil d / 2 \rceil}{k} + \binom{\lfloor d / 2 \rfloor}{k}}.
    \]
\end{theorem}

\begin{corollary} \label{corollary: nikulin's inequality}
	Consider a simple convex \(d\)-dimensional polytope, \(d \geqslant 3\). The mean edge number of its \(2\)-dimensional faces is less than \[A_d^{(1, 2)} = \begin{cases} \frac{4(d - 1)}{d - 2}, & \text{\(d\) is even}; \\ \frac{4d}{d - 1}, & \text{\(d\) is odd}. \end{cases}\]
\end{corollary}

Let \(P \subset \hyperbolic^d\) be a compact Coxeter polytope whose Coxeter diagram \(S\) contains no Lann\'er subdiagrams of order \(\geqslant 3\). Denote by \(a_l\) the number of its \(l\)-dimensional faces and by \(a_{2, k}\) the number of its \(2\)-dimensional \(k\)-gonal faces. Note that the absence of high-order Lann\'er subdiagrams implies that \(a_{2, 3} = 0\).

\begin{lemma} \label{lemma: few faces}
	Under these assumptions, \(a_{2, 4} \leqslant a_0 \cdot (d - 1)\).
\end{lemma}

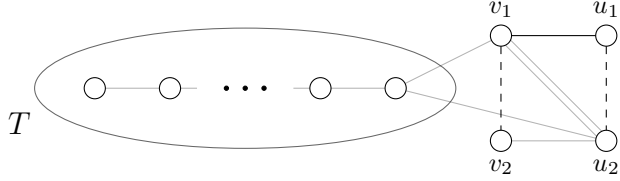
\begin{center}\begin{coxeter}[2]
    \draw (-1.5, -.1) node[below] {\Large $T$};
    \draw[opacity=.6] (0, 0) ellipse (1.4 and .4);
    
    \node{v1}{(3.4, .7)}{$v_1$}
    \node[-90]{v2}{(3.4, -.7)}{$v_2$}
    \node{u1}{(4.8, .7)}{$u_1$}
    \node[-90]{u2}{(4.8, -.7)}{$u_2$}
    
    \dline{v1}{v2}{}
    \dline{u1}{u2}{}
    \line{v1}{u1}{}
    
    \node{t1}{(-2, 0)}{}
    \node{t2}{(-1, 0)}{}
    \hnode{t3}{(-.5, 0)}{}
    \hnode{t4}{(.5, 0)}{}
    \node{t5}{(1, 0)}{}
    \node{t6}{(2, 0)}{}

	\tikzset{thickness/.style={opacity=.3}}
	\line{t1}{t2}{}
	\line{t2}{t3}{}
	\dots{t3}{t4}{}
	\line{t4}{t5}{}
	\line{t5}{t6}{}
	
	\line{t6}{v1}{}
	\line{t6}{u2}{}
	\lline{v1}{u2}{}
	\line{v2}{u2}{}
\end{coxeter}\end{center}

\begin{proof}
	Let \(T\) be the subdiagram of the diagram \(S\) that corresponds to a \(4\)-gonal face. There are vertices \(v_1, v_2, u_1, u_2\) of the diagram \(S\) such that the diagrams \(\langle T, v_i, u_j \rangle\) are elliptic for \(i, j \in \{1, 2\}\) and diagrams \(\langle v_1, v_2 \rangle\) and \(\langle u_1, u_2 \rangle\) are Lann\'er diagrams. Since the diagram \(\langle v_1, v_2, u_1, u_2 \rangle\) is not superhyperbolic, then, without loss of generality, we may assume that \([v_1, u_1] \geqslant 1\).
	
	Thus, the elliptic diagram \(\langle T, v_1, u_1 \rangle\) with the edge \(v_1 u_1\) provides an angle of the \(4\)-gonal face. The number of such diagrams is equal to \(a_0\). Every such diagram contains at most \(d - 1\) edges.
\end{proof}

\begin{proof}[Proof of Theorem \ref{theorem: 3-free}]
	Let \(d \geqslant 13\). Assume that there exists a compact hyperbolic Coxeter polytope \(P \subset \hyperbolic^d\) whose Coxeter diagram \(S\) contains no Lann\'er subdiagrams of order \(\geqslant 3\). From Corollary \ref{corollary: nikulin's inequality} it follows that the mean number of vertices in \(2\)-dimensional faces \(\varkappa = \binom{d}{2} \cdot \frac{a_0}{a_2}\) is less than \(\frac{4 \cdot 13}{12} = 4 \frac{1}{3}\). Since \(P\) contains no \(2\)-dimensional triangular faces,
	\[
		a_{2, 4} > \frac{2}{3} \cdot a_2 = \frac{2}{3} \cdot \binom{d}{2} \cdot \frac{a_0}{\varkappa} > \frac{2}{3} \cdot \frac{13 \cdot 12}{2} \cdot \frac{a_0}{\nicefrac{13}{3}} = 12 a_0.
	\]
	On the other hand, Lemma~\ref{lemma: few faces} implies that \(a_{2, 4} \leqslant a_0 \cdot (d - 1) \leqslant 12 a_0\).
\end{proof}

\bibliographystyle{alpha}
\bibliography{refs.bib}

\end{document}